\documentclass[12pt]{amsart}
\usepackage[top=1.2in, left=1.2in, right=1.2in, bottom=1.2in]{geometry}
\usepackage{amsmath}
\usepackage{amsfonts}
\usepackage{amsthm}
\usepackage[normalem]{ulem}
\usepackage{amssymb}
\usepackage{hyperref}
\usepackage{tikz}
\usepackage{tikz-cd}

\newcommand*{\uu}{\mathfrak{u}}
\newcommand*{\C}{\mathbb{C}}
\newcommand*{\R}{\mathbb{R}}
\newcommand*{\T}{\mathbb{T}}
\newcommand*{\sphere}[1]{\mathrm{S}^{#1}}
\newcommand*{\std}{\mathrm{std}}
\newcommand*{\Bcal}{\mathcal{B}}

\newcommand*{\Fcal}{\mathcal{F}}
\newcommand*{\Gcal}{\mathcal{G}}

\newcommand{\isomorphic}{\approx}

\newcommand*{\abs}[1]{|#1|}
\newcommand*{\norm}[1]{\|#1\|}
\newcommand*{\dotprod}[2]{\langle #1 \:\vert\: #2 \rangle}
\newcommand*{\Hermitian}[1]{\mathcal{H}(#1)}
\newcommand*{\matrices}[1]{\mathcal{M}_{#1}(\C)}
\newcommand{\orbit}[1]{\mathcal{O}(#1)}

\newcommand{\restrict}[3]{#1|_{#2}^{#3}}

\newcommand*{\myblockmatrix}[4]{
	\left( \begin{array}{ccc|c}
		& & & \\
		& #1 & & #2 \\
		& & & \\
		\hline
		& #3 & & #4
	\end{array} \right)
}

\DeclareMathOperator{\diag}{diag}
\DeclareMathOperator{\Ad}{Ad}

\DeclareMathOperator{\Span}{Span}
\DeclareMathOperator{\tr}{tr}
\DeclareMathOperator{\proj}{proj}
\DeclareMathOperator{\Forall}{\forall}

\newtheorem{theorem}{Theorem}[section]

\newtheorem{lemma}[theorem]{Lemma}
\newtheorem{definition}[theorem]{Definition}
\newtheorem{proposition}[theorem]{Proposition}
\theoremstyle{remark}
\newtheorem{remark}[theorem]{Remark}

\numberwithin{equation}{section}

\title[Singular fibers of the Gelfand--Cetlin system]{Singular fibers of the Gelfand--Cetlin system on $\mathfrak{u}(n)^*$}
\author[D. Bouloc]{Damien Bouloc}
\address{Institut de Math\'ematiques de Toulouse, UMR5219 \\
Universit\'e Paul Sabatier \\
118 route de Narbonne \\
31062 Toulouse, France}
\email{damien.bouloc@math.univ-toulouse.fr}

\author[E. Miranda]{Eva Miranda}
\address{ Laboratory of Geometry and Dynamical Systems-EPSEB\\ Department of Mathematics-UPC and BGSMath \\                             
Universitat Polit\`ecnica de Catalunya \\
Avinguda del Doctor Mara\~non 44-50 \\
08028, Barcelona, Spain}
\email{eva.miranda@upc.edu}
\thanks{ E. Miranda  is supported by the Catalan Institution for Research and Advanced Studies via an ICREA Academia Prize 2016 and partially supported  by the grants reference number MTM2015-69135-P (MINECO/FEDER) and reference number 2017SGR932 (AGAUR)}
\author[N.T. Zung]{Nguyen Tien Zung}
\address{Institut de Math\'ematiques de Toulouse, UMR5219, 
Universit\'e Paul Sabatier, 118 route de Narbonne,
31062 Toulouse, France}
\thanks{N.T. Zung was partially supported by a research consulting contract with the Center of
Geometry and Physics, IBS, Republic of Korea}
\email{tienzung@math.univ-toulouse.fr}

\date{\today}

\begin{document}
	
	\begin{abstract} In this paper, we show that every singular fiber of the Gelfand--Cetlin system
 on coadjoint orbits of unitary groups is a smooth isotropic submanifold which is diffeomorphic
 to a $2$-stage quotient of a compact Lie group by free actions of two other compact Lie groups. In many cases,
 these singular fibers can be shown to be homogeneous spaces or even diffeomorphic to compact Lie groups. We also give a combinatorial formula for computing the dimensions of all singular fibers,
 and give a detailed description of these singular fibers in many cases,  including the so-called
 (multi-)diamond singularities. These (multi-)diamond singular fibers are degenerate for the Gelfand--Cetlin system, but they are Lagrangian submanifolds diffeomorphic to direct products of special
unitary groups and tori. Our methods of study are based on different ideas involving complex ellipsoids, 
Lie groupoids, and also general ideas coming from the theory of singularities of integrable Hamiltonian
systems.
\end{abstract}
\maketitle

\setcounter{tocdepth}{1}
\tableofcontents

\section{Introduction}

The Gelfand-Cetlin system is a famous integrable Hamiltonian system
on the coadjoint orbits of unitary groups, which was found and studied
by Guillemin and Sternberg in early 1980s \cite{GS83gelfand,GS83collective},
using the so-called Thimm's method of collective motions
\cite{ThimmCollective}, and related to the classical work of Gelfand and Cetlin in representation
theory \cite{GC1950}. A result of  Alekseev and Meinrenken \cite{alekseev} says that 
this system is also equivalent to an integrable system
found by Flaschka and Ratiu \cite{Ratiu}, via the so-called 
Ginzburg-Weinstein transformation.
Compared to many other integrable systems, especially those
arising in classical mechanics and physics (see, e.g., \cite{bolsinovfomenko,zung1996symplectic}),
the Gelfand-Cetlin has some very special topological and geometric properties:

\begin{itemize}
\item Its base space (i.e. the space of connected fibers of the momentum map)
is (affinely equivalent to) a convex polytope, similar to the case of toric systems
(whose base spaces are the so-called Delzant polytopes \cite{delzant}), even though
the system is \emph{not} toric.

\item In fact, the Gelfand-Cetlin momentum map of the system 
(which consists of eigenvalue functions of a chain
of matrices) generates a toric action, but only on a dense open set of the symplectic manifold
in question. This momentum map is \emph{not} globally smooth, though it can be changed into
a smooth momentum map with the same fibers (i.e. preimages) by taking the symmetric functions
of the eigenvalue functions. However, this smooth momentum map does not generate a torus
action. 

\item Unlike the toric case, where the singularities are all elliptic nondegenerate
(in the sense of Vey--Eliasson; see, e.g., \cite{eliasson1990normal,zung1996symplectic,MirandaZung_NF2004,Miranda_CEJM2014, mirandathesis}
for nondegenerate singularities),
the Gelfand-Cetlin system admits many degenerate singularities.

\item The degenerate singular fibers of the Gelfand-Cetlin system are
very peculiar in the sense that
they are all smooth isotropic submanifolds, as will be shown in this paper (see also Cho--Kim--Oh
\cite{CKO_GC2018} where the same result is obtained
by different methods),
while many degenerate singular fibers of other integrable Hamiltonian systems 
are singular varieties. (See, e.g., \cite{bolsinov,bolsinovfomenko, Zung_CorankOne2000,Zung_CRAS2003}
for various results about degenerate singularities).

\item It turns out that the Gelfand-Cetlin systems can be obtained by the method of
toric degenerations, see Nishinou--Nohara--Ueda \cite{NNU_GC2010}. This method, which comes from algebraic geometry,
is now known to generate a lot of ``artificial'' integrable Hamiltonian systems
(see, e.g. Harada--Kaveh \cite{harada2015integrable} and references therein). This toric degeneration
nature of the Gelfand-Cetlin system may be strongly related to its topological
and geometrical particularities.

\end{itemize}

This paper is the result of a project dating back to 2006
to study the singularities of the 
Gelfand-Cetlin system, from the point of view of the general topological
theory of integrable Hamiltonian systems and their singularities. Unlike
some other papers on the subject like \cite{NNU_GC2010,CKO_GC2018},
which are mainly motivated by considerations from algebraic geometry,
our work is mainly motivated by considerations from dynamical systems.
The first results of this project appeared in the form of a PhD thesis 
in 2009 of Iman Alamiddine \cite{alamiddine2009GelfandCeitlin}, who did
it under the supervision of N.T. Zung and with the help of E. Miranda.
The main result of this thesis is a complete description
of a degenerate singular fiber in $\mathfrak{u}(3)^*$, which is a Lagrangian submanifold
diffeomorphic to $S^3$, together with a neighborhood of it (i.e. a symplectic
normal form for the system around this degenerate fiber).
At that time, we conjectured that all the fibers of the Gelfand-Cetlin system
(in any dimension) are smooth isotropic. In particular, N.T. Zung gave a talk on
this subject at IMPA in 2009,  where a method for studying the topology of singular
fibers of the Gelfand-Cetlin system using the complex ellipsoids was presented
and some results were announced\footnote{See \url{http://strato.impa.br/videos/workshop\_geometry/geometry\_070809\_03.avi}; this
talk contains some good ideas and some errors.}. 

Inspired by ideas coming from the Gelfand-Cetlin system,
D. Bouloc proved in 2015 \cite{bouloc2015singular} a similar conjecture 
for all singular fibers of the Kapovich-Millson system of bending flows of 3D polygons \cite{kapovich1996symplectic}, and also of another
similar integrable system studied by Nohara and Ueda
\cite{nohara2014toric} on the 2-Grassmannian manifold.
Namely, he showed that all these fibers are
isotropic submanifolds  (if the ambient symplectic variety itself is a manifold) or orbifolds (in special situations when the ambient symplectic spaces are orbifolds but not manifolds). Remark that these systems of Kapovich-Millson
and Nohara--Ueda can also be obtained by toric degenerations
(see Foth--Hu \cite{FH_ToricDegeneration2005}). 

Encouraged by the results of \cite{bouloc2015singular}, we have a more general conjecture about the singular fibers of integrable systems which can be obtained via toric degenerations, and have worked out the case of 
Gelfand-Cetlin system for the present paper. In particular,
we will show in this paper that 
every singular fiber of the Gelfand--Cetlin system
 on coadjoint orbits of unitary groups is a smooth isotropic submanifold which is diffeomorphic
 to a $2$-stage quotient of a compact Lie group by free actions of two other compact Lie groups
 (Theorem \ref{t:GC_fibers_are_manifolds} and Corollary 
 \ref{cor:isotropic}). In many cases,
 these singular fibers can be shown to be homogeneous spaces or even diffeomorphic to compact Lie groups. We also give a combinatorial formula for computing the dimensions of all singular fibers (Proposition \ref{prop:dimension}), 
 and give a detailed description of these singular fibers in many cases,  including the so-called
 (multi-)diamond singularities. These (multi-)diamond singular fibers are degenerate for the Gelfand--Cetlin system, but they are Lagrangian submanifolds diffeomorphic to direct products of special
unitary groups and tori. 

We remark that Cho, Kim and Oh in a recent preprint \cite{CKO_GC2018} 
already proved the smooth isotropic character of the
singular fibers of the Gelfand-Cetlin system 
and gave a combinatorial formula for the dimensions of these fibers.
We found out \cite{CKO_GC2018} by chance during the preparation of our
paper. Their paper and ours are independent, and 
complement each other,
because the motivations  are completely different 
(Cho, Kim and Oh came to the problem from the
point of view of pure symplectic geometry and 
mirror symmetry,
while our project was motivated by problems coming from
dynamical systems), and the methods used are also very
different. In particular, in our work we use the variety of 
complete flags of complex ellipsoids
which is not present in \cite{CKO_GC2018}. It is precisely
a kind of duality between such a variety of complete flags of 
complex ellipsoids and a coadjoint orbit of the unitary group 
(see Figure \ref{fig:cd_change_of_pov}) that gives us a geometric
description of the fibers of the Gelfand-Cetlin system. 

We notice that some related partial results on the topology of collective integrable systems 
have been obtained by Lane in \cite{lane}.
We remark also that, even though the singularities of the Gelfand-Cetlin system are rather special from the point of 
view of general integrable Hamiltonian systems, there are still many similarities with other singularities that
we encountered before. In particular, there is still a topological decomposition into direct products of simpler
singularities, as will be seen in Section 5. One can also talk about the (real) \textit{toric degree} of these singularities 
(see \cite{Zung_Integrable2016} and references therein for the notion of toric degree), a topic that we will
not discuss in this paper. 

\subsection*{Organization of this paper} 

The rest of this paper is organized as follows:

In section 2 we give the basic notions of Gelfand--Cetlin systems. 

In Section 3 we construct a variety of
complete flags of complex ellipsoids, which is
dual to a symplectic phase space of the Gelfand-Cetlin
system (i.e., a coadjoint oorbit of the unitary group),
and ``move'' the Gelfand-Cetlin from the coadjoint orbit
to a ``dual system'' on the variety of complete flags of complex ellipsoids. The ``dual Gelfand-Cetlin momentum map'' on this ``dual Grassmannian'' also consists of eigenvalue functions of appropriate matrices, and the two
momentum maps (the Gelfand-Cetlin momentum map and its
``dual map'' have the same image). 

In Section 4 we study the symmetry group of an ellipsoid flag with fixed eigenvalues, which gives a new geometrical interpretation of the Gelfand--Cetlin system, and use this machinery to define the symmetry groupoid of ellipsoid flags. These objects yield a good understanding of the geometry of the Gelfand--Cetlin system: The main result of this section is Theorem \ref{t:GC_fibers_are_manifolds} where we prove that the fibers of Gelfand--Cetlin system are smooth submanifolds and identify these manifolds with a quotient manifold constructed using this symmetry groupoid of ellipsoid flags. 

Based on the results of Section 3 and Section 4, in Section 5 we give a combinatorial formula for the dimensions of the fibers of the Gelfand-Cetlin system, 
and show a necessary and sufficient condition for a fiber
to be of maximal possible dimension, or equivalently, to
be a Lagrangian submanifold (Proposition \ref{prop:dimension}). We also give many concrete examples
of the fibers, together with their topological description. In particular, there is an interesting
family of degenerate singular fibers, which we call 
\textit{(multiple-)diamond singularities} because the equalities
in the corresponding Gelfand-Cetlin triangles form 
``diamonds'', and which are Lagrangian submanifolds
diffeomorphic to compact Lie groups of the type
\begin{equation}
SU(l_1) \times \hdots \times SU(l_s) \times \mathbb{T}^{N- \sum (l_i^2+1)},
\end{equation}
where $s$ is the number of diamonds and 
$l_1,\hdots, l_s$ are their sizes. A particular case
of diamond singularities worked out in detail in this section is the case of a degenerate fiber diffeomorphic
to the 3-dimensional sphere $S^3 \approx  SU(2) $ in $\mathfrak{u}(3)^*$, which has been studied
before in Alamiddine's thesis \cite{alamiddine2009GelfandCeitlin}.

Finally, in Section 6 we prove the isotropic character of the fibers of the Gelfand-Cetlin theorem 
(Proposition \ref{prop:IsotropicGC}) by direct computations.  We believe that this is a general phenomenon,
which should be true not only for the Gelfand-Cetlin system and the systems studied by Bouloc
\cite{bouloc2015singular}, but also for many other systems obtained via the toric degeneration method
as well. In this relation, we present in Proposition  \ref{prop:isotropic}
a general sufficient condition for the isotropicness of the singular
fibers of an integrable Hamiltonian system. Unfortunately, this proposition is
still not general enough: there are some degenerate singular fibers of the Gelfand-Cetlin system (especially
those which are Lagrangian) which do not satisfy the conditions of this last proposition. So one will need 
a more general proposition in order to avoid direct case-by-case computations.
    
\section{The Gelfand--Cetlin system on coadjoint orbits of $\uu(n)^\ast$}

\subsection{Coadjoint orbits of $\uu(n)^\ast$} 

Consider the \textit{unitary group}  
\begin{equation}
U(n) = \{ M \in \matrices{n} \mid M M^\ast = I_n = M^\ast M \},
\end{equation}
and its Lie algebra
\begin{equation}
\uu(n) = \{ B \in \matrices{n} \mid B + B^\ast = 0 \}. 
\end{equation}

The dual space $\uu(n)^\ast$ is identified with the space $\Hermitian{n} = \sqrt{-1} \uu(n)$ of \textit{Hermitian matrices} of size $n$ via the map $\varphi : \Hermitian{n} \rightarrow \uu(n)^\ast$ defined for all $A \in \Hermitian{n}$ and $B \in \uu(n)$ by
\begin{equation}
\label{eqn:SymplecticForm}
\varphi(A)(B) = - \sqrt{-1}  \tr(AB). 
\end{equation}

Under this identification, the coadjoint representation of $U(n)$ is simply the matrix conjugation:  $\Ad_C^\ast(A) = CAC^\ast$ for all $C \in U(n)$ and $A \in \Hermitian{n}$.
The spectrum of a Hermitian matrix is real.
Therefore in this paper, by a coadjoint orbit of $U(n)$ we always mean its 
identification with the set $\orbit{\lambda} \subset \Hermitian{n}$ 
of all Hermitian matrices with a fixed real spectrum $\lambda = (\lambda_1, \lambda_2, \hdots,\lambda_n) \in \mathbb{R}^n$ with $\lambda_1 \geq \lambda_2 \geq \hdots \geq \lambda_n$.

Each coadjoint orbit $\orbit{\lambda}$ is equipped with a natural symplectic form, 
called the \textit{\textbf{Kirillov--Kostant--Souriau form}}, of which we recall here the expression. Note that if $A \in \orbit{\lambda}$, then the tangent space of the coadjoint orbit at $A$ can be written as
\begin{equation}
T_A \orbit{\lambda} = \{ [B, A] \mid B \in \uu(n) \},
\end{equation}
and the Kirillov--Kostant--Souriau form $\omega$ on $\orbit{\lambda}$ take the form 
\begin{equation}\omega_A([H_1, A], [H_2, A]) = \sqrt{-1}  \tr(A [H_1, H_2]) 
\end{equation}
for $H_1, H_2 \in \uu(n)$. Moreover, each coadjoint orbit $\orbit{\lambda}$ 
is diffeomorphic to a model given by the following proposition.

\begin{proposition}[{\cite[Prop. II.1.15]{audin2012torus}}]
\label{p:model_for_coadjoint_orbits}
	Suppose that the spectrum $\lambda$ consists of $k$ distinct eigenvalues 
of respective multiplicities $n_1, \dots, n_k$, that is
\begin{equation}
\underbrace{\lambda_1 = \cdots = \lambda_{d_1}}_{n_1} > 
    \underbrace{\lambda_{d_1 + 1} = \cdots = \lambda_{d_2}}_{n_2} > \cdots > \underbrace{\lambda_{d_{k-1} + 1} = \cdots = \lambda_n}_{n_k}. 
\end{equation}
	Then the coadjoint orbit $\orbit{\lambda}$ is diffeomorphic to the 
    homogeneous space $U(n)/U(\lambda)$, where
\begin{equation} 
U(\lambda) \approx U(n_1) \times \cdots \times U(n_k) 
\end{equation}	is the subgroup of block-diagonal unitary matrices of size $n$, with diagonal blocks of respective sizes $n_1, \dots, n_k$.
\end{proposition}

From the above proposition we obtain that the orbit $\orbit{\lambda}$ has dimension $n^2 - (n_1^2 + \cdots + n_k^2)$. In particular, we say that $\lambda$ (and the corresponding orbit $\orbit{\lambda}$) is \emph{generic} if all the eigenvalues are distinct, that is
\begin{equation}\lambda_1 > \cdots > \lambda_n. \end{equation}
In this case, $U(\lambda)$ is a $n$-dimensional torus $\T^n = U(1)^n$ and the generic orbit $\orbit{\lambda}$ has maximal dimension $n^2 - n = n(n-1)$.

\subsection{The Gelfand--Cetlin system}

Fix a spectrum $\lambda = (\lambda_1 \geq \cdots \geq \lambda_n)$. For any matrix $A \in \orbit{\lambda}$ and any integer $1 \leq k \leq n$, denote by $A_k$ the upper-left submatrix of $A$ of size $k \times k$. Since $A_k$ is a Hermitian matrix of size $k$, it admits $k$ real eigenvalues
\begin{equation}F_{1,k}(A) \geq \cdots \geq F_{k,k}(A). \end{equation}

The family of functions
\begin{equation}
F = \{ F_{i,j} : \orbit{\lambda} \rightarrow \R \mid 1 \leq i \leq j < n \}
\end{equation}
satisfies the \textit{\textbf{Gelfand--Cetlin triangle (diagram) of inequalities}} \cite[Prop. 5.3]{GS83gelfand} shown in Figure~\ref{f:gelfand_cetlin_diagram}.

\begin{figure}[!ht]
	\centering
	\begin{tikzpicture}[y=-1cm, x=0.8cm]
	\newcommand{\nodet}[2]{\node at (#1) {$#2$};}
	\newcommand{\nodel}[2]{\node at (#1, 0) {$\lambda_{#2}$};}
	\newcommand{\nodem}[3]{\node at (#1) {$F_{#3, #2}$};}
	\newcommand{\nodef}[1]{\node[rotate=-45] at (#1) {$\geq$};}
	\newcommand{\noder}[1]{\node[rotate=45] at (#1) {$\geq$};}
	
	\nodel{0}{1} \nodet{1,0}{\geq} \nodel{2}{2} \nodet{3,0}{\geq} \nodel{4}{3} \nodet{5,0}{\geq} \nodet{6,0}{\cdots}	\nodet{7,0}{\geq} \nodel{8}{n-1} \nodet{9,0}{\geq} \nodel{10}{n}
	
	\nodem{1,1}{n-1}{1} \nodem{3,1}{n-1}{2} \nodem{5,1}{n-1}{3} \nodet{7,1}{\cdots} \nodem{9,1}{n-1}{n-1}
	\nodef{0.5,0.5} \noder{1.5,0.5} \nodef{2.5,0.5} \noder{3.5,0.5} \nodef{4.5,0.5} \noder{5.5,0.5} \nodef{8.5,0.5} \noder{9.5,0.5} 
	\nodem{2,2}{n-2}{1} \nodem{4,2}{n-2}{2} \nodet{6,2}{\cdots} \nodem{8,2}{n-2}{n-2}
	\nodef{1.5,1.5} \noder{2.5,1.5} \nodef{3.5,1.5} \noder{4.5,1.5} \nodef{7.5,1.5} \noder{8.5,1.5}
	\node[rotate=-45] at (3, 3) {$\cdots$}; \node[rotate=45] at (7, 3) {$\cdots$};
	\nodef{3.5,3.5} \noder{4.5,3.5} \nodef{5.5,3.5} \noder{6.5,3.5}
	\nodem{4,4}{2}{1} \nodem{6,4}{2}{2}
	\nodef{4.5,4.5} \noder{5.5,4.5}
	\nodem{5,5}{1}{1}
	\end{tikzpicture}
	\caption{The Gelfand--Cetlin triangle (diagram).}
	\label{f:gelfand_cetlin_diagram}
\end{figure}
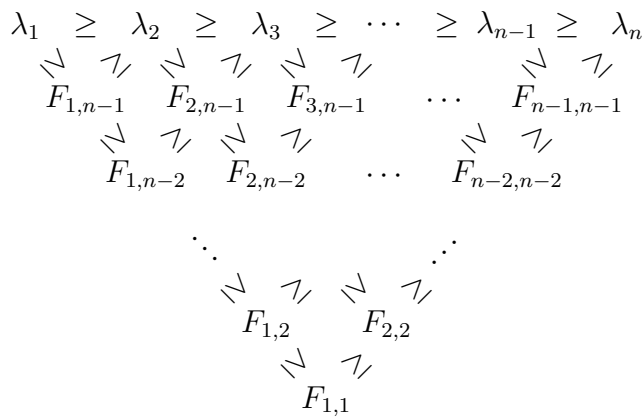
Moreover, these functions commute pairwise under the Poisson bracket induced by the Kirillov--Kostant--Souriau form on $\orbit{\lambda}$, and we have the following theorem due to
Guillemin and Sternberg \cite{GS83collective}:

\begin{theorem}[{\cite{GS83collective}}]
	On any coadjoint orbit $\orbit{\lambda}$, the non-constant functions in $F$ define a completely integrable Hamiltonian system. The singular values of this system are the values for which there is a non-trivial equality in the Gelfand--Cetlin diagram on Figure \ref{f:gelfand_cetlin_diagram}.
\end{theorem}

Indeed, the family $F$ contains $n(n-1)/2$ functions, which is precisely half the dimension of a generic coadjoint orbit. If the coadjoint orbit $\orbit{\lambda}$ is not generic because some of the eigenvalues $\lambda_1, \dots, \lambda_n$ are equal, then the Gelfand--Cetlin inequalities imply that some of the functions in $F$ are constant. But one can check that the number of remaining non-constant functions is, again, half the dimension of $\orbit{\lambda}$. 

Remark that we can also define the functions $F_{1,n}, \hdots, F_{n,n}$
on $\mathfrak{u}(n)^*$
in the same way; they will be constant on $\orbit{\lambda}$:
$F_{1,n} = \lambda_1, \hdots, F_{n,n} = \lambda_n$.

By abusing the notation, for any given spectrum $\lambda$ we will call \textit{\textbf{momentum map}} of the Gelfand--Cetlin system, or \textbf{\textit{Gelfand-Cetlin map}}, the 
following so-called \textit{collective function}, which will be denoted by the same letter $F$:
is given:
\begin{equation}
F = (F_{i,j})_{1 \leq i \leq j < n} : 
\orbit{\lambda} \rightarrow \R^{n(n-1)/2}.
\end{equation}

Even though this map has redundant components when $\lambda$ is non-generic, it has the same level sets as the usual momentum map of the Gelfand--Cetlin integrable system, which is enough for the matter of the present paper.

\section{Geometric interpretation of the fibers}

Fix an ordered real spectrum $\lambda = (\lambda_1 \geq \dots \geq \lambda_n)$, and denote by 
\begin{equation}
D_\lambda = \diag(\lambda_1, \dots, \lambda_n) \in \orbit{\lambda}
\end{equation}
the diagonal matrix with entries $\lambda_1, \dots, \lambda_n$. We have a natural projection $\pi : U(n) \rightarrow \orbit{\lambda}$ given by 
\begin{equation}
\pi(C) = C D_\lambda C^\ast. 
\end{equation}

\subsection{Partial and complete flag manifolds}

Recall that a \emph{flag} in $\C^n$ is a sequence of vector subspaces
\begin{equation}
V^\bullet = ( \{ 0 \} = V^0 \subset V^1 \subset \cdots \subset V^k = \C^n ) 
\end{equation}
with increasing dimensions $0 = d_0 < d_1 < \cdots < d_k = n$ (we will often omit $V^0$). The $k$-tuple $(d_1, \dots, d_k)$ is called the \emph{signature} of the flag $V^\bullet$. A flag is \emph{complete} if $k = n$ and $(d_1, d_2, \dots, d_k) = (1, 2, \dots, n)$, otherwise it is \emph{partial}.

Fix a signature $d = (d_1, \dots, d_k)$. One can associate to any flag $V^\bullet$ with signature $d$ a basis $(u_1, \dots, u_n)$ of $\C^n$ such that for any $1 \leq i \leq k$, the space $V^i$ is generated by the first $d_i$ vectors of this basis. Up to a Gram--Schmidt process, the basis $(u_1, \dots, u_n)$ can be supposed unitary, we then identify it with the matrix $C \in U(n)$ with columns $u_1, \dots, u_n$. Conversely, any matrix $C \in U(n)$ represents a unique flag $V^\bullet$ 
with signature $d$ such that for any $1 \leq i \leq k$, $V^i$ 
is generated by the first $d_i$ columns of $C$. 

Note that two matrices $C_1, C_2$ represent the same flag if and only if there exists a block diagonal matrix $P \in U_d = U(n_1) \times \cdots \times U(n_k)$ such that $C_2 = C_1 P$, where $n_i = d_i - d_{i-1}$ (with convention $d_0 = 0$). It is then standard to identify the set of all flags with signature $d$ in $\C^n$ with the homogeneous space
\begin{equation}
\Fcal_d = U(n)/U_d =  U(n)/(U(n_1) \times \cdots \times U(n_k)). 
\end{equation}

In particular, the set of complete flags in $\C^n$ is identified with $U(n)/\T^n$, 
where $\mathbb{T}^n = U(1) \times \cdots \times U(1)$ is the usual $n$-dimensional torus.

\begin{remark}
	In particular, Proposition~\ref{p:model_for_coadjoint_orbits} states that every coadjoint orbit $\orbit{\lambda}$ is diffeomorphic to a flag manifold $\Fcal_d$, with signature $d$ determined by the redundancies among the values $\lambda_1, \dots, \lambda_n$.
\end{remark}

\subsection{Complex ellipsoids}

Let $V$ be a finite-dimensional (complex) vector space with a Hermitian product $\dotprod{.}{.}$ (the convention chosen in this paper for a Hermitian product is to be linear in the first variable and anti-linear the second variable). Recall that a linear transformation $\alpha : V \rightarrow V$ is called \emph{Hermitian} if $\dotprod{\alpha(v_1)}{v_2} = \dotprod{v_1}{\alpha(v_2)}$ for any $v_1, v_2 \in V$,
i.e., $\alpha^\ast = \alpha$. A Hermitian transformation is diagonalizable in a unitary basis: there exists a basis $(v_1, \dots, v_k)$ of $V$, with $\dotprod{v_i}{v_j} = \delta_{i,j}$, such that $\alpha(v_i) = \gamma_i v_i$. Moreover, its eigenvalues $\gamma_1, \dots, \gamma_k$ are real numbers. In particular, for any $v = x_1 v_1 + \cdots + x_k v_k$ in $V$ we have
\begin{equation}\dotprod{\alpha(v)}{v} = 1 \iff \gamma_1 \abs{x_1}^2 + \cdots + \gamma_k \abs{x_k}^2 = 1. 
\end{equation}

By analogy with the real Euclidean case, we will set the following definition:
\begin{definition}
	A \textbf{complex ellipsoid} in $V$ is a subset of the form
\begin{equation}
E_\alpha = \{ v \in V \mid \dotprod{\alpha(v)}{v} = 1 \}
\end{equation}
	where $\alpha : V \rightarrow V$ is a positive definite Hermitian transformation (recall that $\alpha$ is positive definite if $\dotprod{\alpha(v)}{v} > 0$ for any $v \neq 0$, or equivalently if all its eigenvalues are positive).
    If $(v_1, \dots, v_k)$ is a basis of eigenvectors 
for $\alpha$ and $\gamma_1, \dots, \gamma_k$ are the associated eigenvalues,
then we say that $E_\alpha$ has axes $\mathbb{C}v_1, \dots, \mathbb{C}v_k$ and 
radii $1/\sqrt{\gamma_1}$, $\dots$, $1/\sqrt{\gamma_k}$.
\end{definition}

When $V = \C^k$ and $A \in \Hermitian{k}$, we simply write
\begin{equation}
E_A = \{ x \in \C^k \mid \dotprod{Ax}{x} = 1 \}. 
\end{equation}

The complex ellipsoids satisfy the following immediate properties:
\begin{lemma}
	\label{l:properties_ellipsoid}
	Let $\alpha, \beta$ be positive definite Hermitian 
    tranformations of $V$, and $\phi : V \rightarrow W$ a unitary map. Then:
	\begin{enumerate}
		\item for any $v \neq 0$ in $V$, there exists $t > 0$ such that $tv \in E_\alpha$,
		\item $E_\alpha = E_\beta$ if and only if $\alpha = \beta$,
		\item $\phi(E_\alpha) = E_{\phi \circ \alpha \circ \phi^{-1}}$,
		\item if $V=W$, then $\phi$ preserves $E_\alpha$ if and only if $\alpha$ and $\phi$ commute.
	\end{enumerate}
\end{lemma}

\begin{proof}
	Fix $v \neq 0$. Since $\alpha$ is positive definite, $\dotprod{\alpha(v)}{v} > 0$. Then $t = 1/\sqrt{\dotprod{\alpha(v)}{v}}$ is well-defined and we have:
	\[ \dotprod{\alpha(tv)}{tv} = t^2 \dotprod{\alpha(v)}{v} = 1 \]
	which proves (1). But if $E_\alpha = E_\beta$, then we also have $\dotprod{\beta(tv)}{tv} = 1$, hence $\dotprod{\beta(v)}{v} = 1/t^2 = \dotprod{\alpha(v)}{v}$. That is
	\[ \dotprod{(\beta-\alpha)v}{v} = 0 \]
	for all $v \in \C^n$. Since $\beta-\alpha$ is again Hermitian, it is diagonalizable. But the above condition implies that all its eigenvalues are zero, so we conclude that $\alpha=\beta$, which proves (2). The equality (3) comes from:
	\[ \dotprod{\alpha(v)}{v} = \dotprod{\phi \circ \alpha(v)}{\phi(v)} = \dotprod{\phi \circ \alpha \circ \phi^{-1}(\phi(v))}{\phi(v)}. \]
	Then (4) follows from (2) and (3).
\end{proof}

The following proposition deals with the intersection of a complex ellipsoid with a lower dimensional vector subspace.

\begin{proposition}
	\label{p:intersection_ellipsoid-subspace}
	Let $E_\alpha$ be a complex ellipsoid in $V$, and $W, W'$ two linear subspaces of $V$. Then:
	\begin{enumerate}
		\item $E_\alpha \cap W$ is a complex ellipsoid in $W$: there exists a positive definite Hermitian map $\beta : W \rightarrow W$ such that $E_\alpha \cap W = E_\beta$.
		\item In particular, if $V = \C^n$, $W = \C^k \times \{ 0 \}^{n-k}$ then $E_A \cap W = E_{A_k}$.
		\item If $E_\alpha \cap W = E_\alpha \cap W'$, then $W = W'$.
	\end{enumerate}
\end{proposition}

\begin{proof}
	For (2), it suffices to remark that $\dotprod{A i(x)}{i(x)}_n = \dotprod{A_k x}{x}_k$ for every $x \in \C^k$, where $i : \C^k \rightarrow \C^k \times \{ 0 \}^{n-k}$ is the canonical identification. Now for (1), choose $\phi : \C^n \rightarrow V$ unitary such that $\phi(\C^k \times \{ 0 \}^{n-k}) = W$. Define $A \in \Hermitian{n}$ by $Ax = \phi^{-1} \circ \alpha \circ \phi(x)$, hence $w \in E_\alpha$ if and only if $\phi^{-1}(w) \in E_A$. Now define $\beta : W \rightarrow W$ by $\beta(w) = \phi(A_k \phi^{-1}(w))$. We obtain $w \in E_\alpha \cap W$ if and only if $\phi^{-1}(w) \in E_A \cap (\C^k \times \{ 0 \}^{n-k})$, which is equivalent to $w \in E_\beta$.
	
	For (3), consider $w \in W$. By Lemma~\ref{l:properties_ellipsoid}, there exists $t > 0$ such that $tw \in E_\alpha$. If $E_\alpha \cap W = E_\alpha \cap W'$ then, $tw$, and hence $w$, lie in $W'$ and so $W \subset W'$. With a symmetric argument we conclude that $W = W'$.
\end{proof}

The above properties motivate the following definitions:
\begin{definition}
	An \textbf{ellipsoid flag} in $\C^n$ is a triple $(E^\bullet, V^\bullet, A)$ where:
	\begin{itemize}
		\item[i)] $V^\bullet$ is a (vector space) flag in $\C^n$,
		\item[ii)] $A$ is a positive definite Hermitian matrix of size $n$,
		\item[iii)] $E^\bullet = ( E^1 \subset E^2 \subset \cdots \subset E^k = E_A )$ 
        is the increasing sequence of ellipsoids defined by $E^k = E_A \cap V^k$.
	\end{itemize}
	The \textbf{signature} of $(E^\bullet, V^\bullet, A)$ is the signature 
    of $V^\bullet$, and we say that $(E^\bullet, V^\bullet, A)$ is \textbf{complete}
    if $V^\bullet$ is complete.
\end{definition}

We will sometimes denote the ellipsoid flag 
$(E^\bullet, V^\bullet, A)$ by $E^\bullet = E_A \cap V^\bullet$, 
or simply by $E^\bullet$.

Recall that, for each $1 \leq k \leq n$, the complex ellipsoid 
$E^k = E_{\alpha_k}$ is defined by a unique positive definite Hermitian 
map $\alpha_k : V^k \rightarrow V^k$. For brevity, we will say that the flag 
$E^\bullet$ is defined by the family $\alpha_\bullet = (\alpha_1, \dots, \alpha_n)$.

\begin{definition}
	Let $(E^\bullet, V^\bullet, A)$ be a complete ellipsoid flag.
We call \textbf{eigenvalues} of $E^\bullet$ the $(n(n+1)/2)$-tuple
\begin{equation}
\Gamma(E^\bullet) = ( \gamma_{i,j}(E^\bullet) \mid 1 \leq j \leq n,\  1 \leq i \leq j ),
\end{equation}
	where for any $1 \leq j \leq n$,
\begin{equation}
\gamma_{1,j}(E^\bullet) \geq \gamma_{2,j}(E^\bullet) \geq \cdots \geq \gamma_{j,j}(E^\bullet) 
\end{equation}
	are the eigenvalues of the defining maps $\alpha_j : V^j \rightarrow V^j$ 
    of $E^\bullet$.
\end{definition}

Fix a positive spectrum $\lambda = (\lambda_1 \geq \dots \geq \lambda_n > 0)$ and consider a matrix $A \in \orbit{\lambda}$. Let $C \in U(n)$ such that $A = \pi(C) = C D_\lambda C^\ast$. As above, for $1 \leq k \leq n$ denote by $A_k$ the upper-left submatrix of $A$ and by $i_k : \C^k \hookrightarrow \C^n$ the canonical inclusion. Consider
\begin{equation}
V_{\std}^\bullet = ( i_1(\C) \subset i_2(\C^2) \subset \cdots \subset i_n(\C^n) )
\end{equation}
the standard complete flag of $\C^n$, and $V_C^\bullet$ its image by the linear map
$x \mapsto C^\ast.x$. By Proposition~\ref{p:intersection_ellipsoid-subspace}, the complete ellipsoid flag $(\tilde{E}_A^\bullet, V_{\std}^\bullet, A)$, and hence its image under $\Phi$ denoted by 
flag $(E_C^\bullet, V_C^\bullet, D_\lambda)$, have eigenvalues
\begin{equation}
F(A) = \{ F_{i,j}(A) \mid 1 \leq i \leq j \leq n \}.
\end{equation}

Note that $V^\bullet_C=Q(C^\ast)$ where $Q : U(n) \rightarrow \Fcal$ is the 
map which sends each nondegenerate matrix to the flag defined by its columns, 
so the ellipsoid flag $E_C^\bullet$ depends on the diagonalization $C$ chosen for $A$.

Introduce the eigenvalue map
\begin{equation}
\begin{array}{rrcl}
	\Gamma_\lambda : & \Fcal & \longrightarrow & \R^N \\
		& V^\bullet & \longmapsto & \Gamma(E_{D_\lambda} \cap V^\bullet)
\end{array} 
\end{equation}
and denote by $t : U(n) \rightarrow U(n)$ the involution $t(C)= C^\ast = C^{-1}$.
Then the above remarks are summed up in the commutative diagram given in Figure~\ref{fig:cd_change_of_pov}. 

\begin{figure} [!ht]
	\centering
	\begin{tikzcd}
		U(n) \arrow[d, "\pi: C \mapsto CD_\lambda C^\ast"'] \arrow[rr,  "t: C \mapsto C^\ast"'] & & U(n) \arrow[d, "Q:\ \text{flag generated by the columns}"] \\[1.5em]
		\orbit{\lambda} \arrow[dr, "F:\ \text{Gelfand-Cetlin map}"'] & & \Fcal \arrow[dl, "\Gamma_\lambda: V^\bullet \longmapsto  \Gamma(E_{D_\lambda} \cap V^\bullet)"] \\
		 & \R^N & 
	\end{tikzcd}	
	\caption{From the Gelfand--Cetlin system to the map $\Gamma_\lambda$ 
    on the flag manifold $\Fcal$.}
	\label{fig:cd_change_of_pov}
\end{figure}

\begin{definition} \label{defn:fiber}
This this paper, by a \textbf{\textit{fiber}},
we will mean either a preimage of the map $F$ on $\orbit{\lambda}$
(i.e. a fiber of the Gelfand-Cetlin system on a given coadjoint orbit), or a 
preimage of the map $\Gamma_\lambda$ in the  diagram in Figure~\ref{fig:cd_change_of_pov}.
\end{definition}

\section{The geometry of ellipsoid flags with fixed eigenvalues}

Fix $\lambda = (\lambda_1 \geq \dots \geq \lambda_n > 0)$ a positive spectrum. As before, denote by $D_\lambda$ the diagonal matrix with entries $\lambda_1, \dots, \lambda_n$, and by $E_{D_\lambda}$ the corresponding ellipsoid in $\C^n$.

\subsection{Symmetry group of an ellipsoid flag} 

Let $V^\bullet$ in $\Fcal$ be a given complete flag and 
$E^\bullet = E_{D_\lambda} \cap V^\bullet$ the ellipsoid flag obtained by  intersecting $V^\bullet$ with the ``standard'' ellipsoid $E_{D_\lambda}$.
 For a fixed $1 \leq k \leq n$, let us describe the subgroup $G_k$ 
 of $U(V^k)$ consisting of those elements which the $k$-dimensional 
 complex ellipsoid $E^k  = E_{\alpha_k} \subset V^k$.
We recall the following result from linear algebra.

\begin{lemma}
	\label{l:commuting_endomorphims}
	Let $V$ be a finite dimensional inner product space, $\alpha : V \rightarrow V$ a positive definite Hermitian transform, and $V = W_1 \oplus \cdots \oplus W_r$ the decomposition of $V$ into eigenspaces of $\alpha$.
	Then the subgroup of unitary transformations $\varphi \in U(V)$ which commute 
    with $\alpha$ is exactly
\begin{equation}
G = U(W_1) \oplus \cdots \oplus U(W_r). 
\end{equation}
\end{lemma}

\begin{remark}
	In the above expression, by $\varphi = \varphi_1 \oplus \cdots \oplus \varphi_r$ with $\varphi_i \in U(W_i)$ we mean the map $\varphi : V \rightarrow V$ defined by
\begin{equation}
\varphi(v) = \varphi_1(v_1) \oplus \cdots \oplus \varphi_r(v_r) 
\end{equation}
	for any $v = v_1 \oplus \cdots \oplus v_r \in W_1 \oplus \cdots \oplus W_r$.
\end{remark}

This implies the following:
\begin{proposition}
	\label{p:structure_G_k}
	The subgroup of unitary transformations of $V^k$ that preserves $E^k$ is exactly
\begin{equation} G_k = U(W_1) \oplus \cdots \oplus U(W_r) \subset U(V^k), \end{equation}
	where $V^k = W_1 \oplus \cdots \oplus W_r$ is the decomposition of $V^k$ into eigenspaces of $\alpha_k$.
	
	Each $W_i$ has dimension $n_{i}$ determined by the eigenvalues
\begin{equation} \underbrace{\gamma_{1,k} = \cdots = \gamma_{d_1,k}}_{n_1} > \underbrace{\gamma_{d_1 + 1,k} = \cdots = \gamma_{d_2,k}}_{n_2} > \cdots > \underbrace{\gamma_{d_{r-1} + 1,k} = \cdots = \gamma_{k,k}}_{n_r} 
\end{equation}
of $\alpha_k$. (``Horizontal'' equalities in the Gelfand--Cetlin diagram).
\end{proposition}

\begin{proof}
	The first part of the proposition is immediate: by Lemma~\ref{l:properties_ellipsoid}, $\phi \in U(V^k)$ preserves $E_{\alpha_k}$ if and only if $\phi$ commutes with $\alpha_k$. We conclude using Lemma~\ref{l:commuting_endomorphims}.
	Since $V^\bullet$ is in $\Fcal(c)$, $\alpha_k$ has eigenvalues 
    $\gamma_{1,k} \geq \cdots \geq \gamma_{k,k}$, so the numbers $n_i$ correspond indeed to the dimensions of the different eigenspaces.
\end{proof}

For later use, we will need the subgroup $H_{k+1} \subset G_{k+1}$ of unitary transformations of $V^{k+1}$ which preserve not only $E^{k+1}$ but also $V^k$ (or equivalently, $E^k$), and $H'_{k+1}$ the subgroup of transformations in $H_{k+1}$ whose restriction to $L_k$ is the identity,
where $L_k$ denotes the orthogonal complement of $V^k$ in $V^{k+1}$.
Note that $\dim L_k = 1$.

\begin{lemma}
\label{l:alpha_k+1_depends_on_alpha_k}
Fix a vector $\ell \in L_k$ of length 1 and write 
$\alpha_{k+1}(\ell) = w \oplus a \ell$ with $w \in V^k$, $a \in \C$. 
Then for any $v \in V^k$ we have
\begin{equation}
\alpha_{k+1}(v) = \alpha_k(v) \oplus \dotprod{v}{w} \ell
\end{equation}
\end{lemma}

\begin{proof}
	Using the decomposition $V^{k+1} = V^k \oplus L_k$, write $\alpha_{k+1}(v) = \beta(v) \oplus \lambda(v)\ell$ with $\beta : V^k \rightarrow V^k$ and $\lambda : V^k \rightarrow \C$. Since $V^k$ and $L_k$ are orthogonal, we have
	\[ \lambda(v) = \dotprod{\alpha_{k+1}(v)}{\ell} = \dotprod{v}{\alpha_{k+1}(\ell)} = \dotprod{v}{w}. \]
	It remains to show that $\beta = \alpha_k$. To do so, remark that for any $v \in V^k \subset V^{k+1}$,
	\[ \dotprod{\alpha_{k+1}(v)}{v} = \dotprod{\beta(v)}{v} \]
	hence $E_{\alpha_k} = E_{\alpha_{k+1}} \cap V^k = E_{\beta}$. By Lemma~\ref{l:properties_ellipsoid}, it follows that $\alpha_k = \beta$.
\end{proof}

We are now able to describe the groups $H'_{k+1}$ and $H_{k+1}$.

\begin{proposition}
	\label{p:structure_H_k+1}
	Let $V^k = W_1 \oplus \cdots \oplus W_r$ be the decomposition of $V^k$ into the eigenspaces of $\alpha_k$, and
\begin{equation} 
V^{k+1} = W_1 \oplus \cdots \oplus W_r \oplus L_k 
\end{equation}
	the induced decomposition of $V^{k+1}$, where $L_k$ is the orthogonal complement of $V^k$ in $V^{k+1}$. Denote by $\proj_{W_i} : V^{k+1} \rightarrow W_i$ the orthogonal projection on $W_i$. Then the group $H_{k+1}$ of 
unitary maps on $V^{k+1}$ which preserve $E^k$ and $E^{k+1}$ is 
the set of elements 
\begin{equation} \phi_1 \oplus \cdots \oplus \phi_r \oplus \xi.\mathrm{id}_{L_k} \in U(W_1) \oplus \cdots \oplus U(W_r) \oplus U(L_k) 
\end{equation}
	such that for any $1 \leq i \leq n$ and $w_i \in \proj_{W_i}(\alpha_{k+1}(L_k))$ we have $\phi_i(w_i) = \xi w_i$.

The subgroup $H'_{k+1}$ is the set of all above elements with $\xi$ equal to $1$.
\end{proposition}

\begin{proof}
	Fix a unit vector $\ell \in L_k$ (the orthonormal complement of $V^k$ in $V^{k+1}$), and write
	\[ \alpha_{k+1}(\ell) = w_1 \oplus \cdots \oplus w_r \oplus a \ell \]
	with $w_i \in W_i$ and $a \in \C$. Then, by Lemma~\ref{l:alpha_k+1_depends_on_alpha_k}, for any $v_i \in W_i$,
	\[ \alpha_{k+1}(v_i) = \alpha_k(v_i) \oplus \dotprod{v_i}{w_1 \oplus \cdots \oplus w_r} \ell = \gamma_{ik} v_i \oplus \dotprod{v_i}{w_i} \ell. \]
	Note that $w_i$ depends of the choice of $\ell \in L_k$, but $\C w_i$ can be determined intrinsically as the subspace $\proj_{W_i}(\alpha_{k+1}(L_k))$.
	
	Let $\phi \in U(V^{k+1})$. Suppose $\phi$ preserves $V^k$, and more precisely $E^k$. Then 
	\[ \phi = \phi_1 \oplus \cdots \oplus \phi_r \oplus \xi.\mathrm{id}_{L_k} \quad \in U(W_1) \oplus \cdots \oplus U(W_r) \oplus U(L_k). \]
	On the other hand $\phi$ preserves $E^{k+1}$ if and only if $\phi$ commutes with $\alpha_{k+1}$.
	For $v_i \in W_i$ we have
	\[ \begin{cases}
		\alpha_{k+1}(\phi(v_i)) = \gamma_{ik} \phi_i(v_i) \oplus \dotprod{\phi_i(v_i)}{w_i} \ell = \gamma_{ik} \phi_i(v_i) \oplus \dotprod{v_i}{\phi_i^\ast(w_i)} \ell, \\
		\phi(\alpha_{k+1}(v_i)) = \gamma_{ik}\phi_i(v_i) \oplus \dotprod{v_i}{w_i}\xi \ell = \gamma_{ik}\phi_i(v_i) \oplus \dotprod{v_i}{\bar{\xi} w_i} \ell.
	\end{cases}
 \]
 Similarly,
 \[ \begin{cases}
	 	\alpha_{k+1}(\phi(\ell)) = \alpha_{k+1}(\xi \ell) = \xi w_1 \oplus \cdots \oplus \xi w_r \oplus \xi a \ell, \\
	 	\phi(\alpha_{k+1}(\ell)) = \phi(w_1 \oplus \cdots \oplus w_r \oplus a \ell) = \phi_1(w_1) \oplus \cdots \oplus \phi_r(w_r) \oplus \xi a \ell.
 \end{cases} \]
	It follows that $\phi$ preserves $E^{k+1}$ if and only if for any $1 \leq i \leq r$, $\phi_i(w_i) = \xi w_i$.
\end{proof}

\begin{remark}
	\label{rem:vertical_equalities}
	The space $\proj_{W_i}(\alpha_{k+1}(L_k))$ is trivial if $\alpha_{k+1}(L_k)$ is orthogonal to $W_i$. In this case, for any $w_i \in W_i$ we have
$\alpha_{k+1}(w_i) = \alpha_k(w_i)$, 
	so $W_i$ is also an eigenspace of $\alpha_{k+1}$ associated to the same eigenvalue, and there must be a vertical equality in the Gelfand--Cetlin diagram
    for $E^\bullet$.
	
	By contraposition, if the eigenvalue $\gamma_{i,k}$ associated to $W_i$ appears only once in the Gelfand--Cetlin diagram, then the space $\proj_{W_i}(\alpha_{k+1}(L_k))$ has necessarily dimension $1$.
	
	However, note that the converse is not true. A vertical inequality in the Gelfand--Cetlin diagram does not imply that the corresponding space $\proj_{W_i}(\alpha_{k+1}(L_k))$ is trivial. 
\end{remark}

\begin{remark}
	If $W'_i$ denotes the orthogonal complement of $W''_i = \proj_{W_i}(\alpha_{k+1}(L_k))$ in $W_i$, then $H_{k+1}$ can be written as the group of all elements of the form
	\begin{multline*}
	(\phi'_1 \oplus \xi.\mathrm{id}_{W''_1}) \oplus \cdots \oplus (\phi'_r \oplus \xi.\mathrm{id}_{W''_r}) \oplus \xi \mathrm{id}_{L_k} \\
	\in (U(W'_1) \oplus U(W''_1)) \oplus \cdots \oplus (U(W'_r) \oplus U(W''_r)) \oplus U(L_k).
	\end{multline*}
	This group is clearly isomorphic to
	\[ U(W'_1) \times \cdots \times U(W'_r) \times U(1) \]
	and each $W'_i$ has codimension at most 1 in $W_i$.
	In particular, the subgroup $H'_{k+1}$ is isomorphic to
	\[ U(W'_1) \times \cdots \times U(W'_r) \]
\end{remark}

Consider the group
\begin{equation}
G(E^\bullet) = G_1 \times \cdots \times G_n, 
\end{equation}
	where for each $1 \leq k \leq n$, $G_k$ is the group of unitary transformations of $V^k$ preserving $E^k = E_{D_\lambda} \cap V^k$.
Consider also the subgroups $H'(E^\bullet) < H(E^\bullet) < G(E^\bullet)$ 
defined by
\begin{equation} H'(E^\bullet) = H'_1 \times \cdots \times H'_n \quad \text{ and } \quad H(E^\bullet) = H_1 \times \cdots \times H_n 
\end{equation}
	where for each $1 \leq k \leq n$, $H_k$ is the group of unitary transformations of $V^k$ preserving both $E^k$ and $E^{k-1}$ (or equivalently, both $E^k$ and $V^{k-1}$) and $H'_k$ is the group of transformations in $H_k$ whose restriction to $(V^{k-1})^{\perp V^{k}}$ is the identity.

\begin{definition}
	\label{def:symmetries_ellipsoid_flag} 
The group $G(E^\bullet)$ is called the \textbf{coarse symmetry group}
of the ellipsoid flag $E^\bullet$.
The quotient manifold
\begin{equation}
S(E^\bullet) = G(E^\bullet)/H'(E^\bullet)
\end{equation}
with respect to the free right $H'(E^\bullet)$-action on $G(E^\bullet)$
defined by
	\begin{equation}
		\label{eq:action_on_symmetry_group}
		\phi \cdot f = (\restrict{(f_2^{-1})}{V^1}{V^1} \circ \phi_1 \circ f_1, \dots, \restrict{(f_n^{-1})}{V^{n-1}}{V^{n-1}} \circ \phi_{n-1} \circ f_{n-1}, \phi_n \circ f_n)
	\end{equation}
for any $\phi \in G(E^\bullet)$, $f \in H'(E^\bullet)$,     
is called the  \textbf{reduced symmetry space} of $E^\bullet$.
\end{definition}

\subsection{Push-forward of flags}

In this subsection, $V^\bullet$ is a complete flag in $\C^n$ and 
$E^\bullet = E_{D_\lambda} \cap V^\bullet$ is its intersection with the standard ellipsoid. 
Fix an element $C^\ast \in U(n)$  such that $Q(C^\ast) = V^\bullet$, i.e. columns $(u_1, \dots, u_n)$ of $C^\ast$
form a unitary basis of $V^\bullet$.

\begin{definition}
	Let $\phi = (\phi_1, \dots, \phi_n)$ in $G(E^\bullet)$.
	\begin{enumerate}
		\item The \textbf{push-forward of $V^\bullet$ by $\phi$} is the complete flag
 \begin{equation}   
        \phi_\ast V^\bullet = ( V_\phi^1 \subset \cdots \subset V_\phi^n = \C^n )         
 \end{equation}   
		defined by
 \begin{equation} 
 \label{eqn:pushV}
V_\phi^k = \phi_n(\phi_{n-1}(\dots\phi_{k}(V^k)\dots))
 \quad  \forall\ 1 \leq k \leq n.
  \end{equation} 
		
		\item The \textbf{push-forward of $C^\ast = (u_1, \dots, u_n)$ by $\phi$} is the family
 \begin{equation} 
 \phi_\ast C^\ast = (u_1^\phi, \dots, u_n^\phi)
   \end{equation}
		defined by
 \begin{equation} 
u^\phi_k = \phi_n(\phi_{n-1}(\dots\phi_{k}(u_k)\dots))  \quad  \forall\ 1 \leq k \leq n.. 
  \end{equation}        
	\end{enumerate}	
\end{definition}

The following proposition shows that the above definition makes sense.

\begin{proposition}
	\label{p:push_forward_commute_with_q}
	For any $\phi \in G(E^\bullet)$ and $C^\ast \in U(n)$ such that 
    $Q(C^\ast) = V^\bullet$:
	\begin{enumerate}
		\item $\phi_\ast C^\ast$ is a unitary matrix,
		\item $Q(\phi_\ast C^\ast) = \phi_\ast V^\bullet$.
	\end{enumerate}
	In other words, the map $Q : U(n) \rightarrow \Fcal$ intertwines the two push-forward operations defined above.
\end{proposition}

\begin{proof}
	Recall that $(u_1, \dots, u_n)$ is a unitary basis. Since the maps $\phi_1, \dots, \phi_n$ are unitary transformations, $u^\phi_1, \dots, u^\phi_n$ are also unit vectors. Moreover, for any $1 \leq i < j \leq n$,
	\begin{align*}
		\dotprod{u^\phi_i}{u^\phi_j} & = \dotprod{\phi_n(\phi_{n-1}(\cdots\phi_i(u_i)\cdots))}{\phi_n(\phi_{n-1}(\cdots\phi_j(u_j)\cdots))} \\
		& = \dotprod{\phi_{j-1}(\cdots\phi_i(u_i)\cdots)}{u_j} \\
		& = 0
	\end{align*}
	since $\phi_{j-1}(\cdots\phi_i(u_i)\cdots)$ lies in $V^{j-1} = \Span(u_1, \dots, u_{j-1})$. It follows that $\phi_\ast C^\ast$ is a unitary matrix. Moreover, it is clear that each $u_i$ lies in $V_\phi^i$ so $\phi_\ast C^\ast$ is a basis of the flag $\phi_\ast V^\bullet$.
\end{proof}

The push-forwards of flags and unitary bases can be rewritten more concisely in terms of applications $\bar{\phi}_1, \dots, \bar{\phi}_n$ defined in the following lemma.

\begin{lemma}
	\label{l:properties_bar_phi}
	Let $\phi \in G(E^\bullet)$.
	For $1 \leq k \leq n$, define $\bar{\phi}_k : V^k \rightarrow V_\phi^k$ recursively by:
\begin{equation}
\begin{cases}
			\bar{\phi}_n = \phi_n, & \\
			\bar{\phi}_k = \restrict{(\bar{\phi}_{k+1})}{V^k}{V^k_\phi} \circ \phi_k &\text{ for } 1 \leq k < n,
		\end{cases} 
\end{equation}
	Then for any $1 \leq k \leq n$:
	\begin{enumerate}
		\item $u^\phi_k = \bar{\phi}_k (u_k)$,
		\item $V_\phi^k = \bar{\phi}_k (V^k) = \bar{\phi}_{k+1}(V^k)$,
		\item $E_{D_\lambda} \cap V_\phi^k = \bar{\phi}_k(E^k) = \bar{\phi}_{k+1}(E^k)$
	\end{enumerate}
	(with the convention that $\bar{\phi}_{n+1}$ is the identity map on $\C^n$).
\end{lemma}

\begin{proof}
	By definition of $\bar{\phi}_k$, for any subset $S \subset V^k$ we have 
	\[ \bar{\phi}_k(S) = \phi_n(\phi_{n-1}(\cdots\phi_k(S)\cdots)). \]
	Taking $S = \{ u_k \}$ proves (1) while $S = V^k = \phi_k^{-1}(V^k)$ proves (2). Now, recall that by definition $\phi_k : V^k \rightarrow V^k$ preserves $E^k = E_{D_\lambda} \cap V^k$. Then for any $\ell \geq k$, by writing
	\[ E^k = E_{D_\lambda} \cap V^k = E_{D_\lambda} \cap V^\ell \cap V^k \]
	we get
	\[ \phi_\ell(E^k) = \phi_\ell(E_{D_\lambda} \cap V^\ell) \cap \phi_\ell(V^k) = E_{D_\lambda} \cap V^\ell \cap \phi_\ell(V^k) = E_{D_\lambda} \cap \phi_\ell(V^k). \]
	Iterating this process, we can show that $\bar{\phi}_\ell(E^k) = E_{D_\lambda} \cap \bar{\phi}_\ell(V^k)$.
	Take $\ell = k$ or $\ell = k+1$ to prove (3).
\end{proof}

The push-forwards defined in this section do not provide a one-to-one correspondence between the group of symmetries $G(E^\bullet)$ and the set of flags with given eigenvalues. Actually, two symmetries give the same push-forwards if the are equal up to an element in the groups $H'(E^\bullet)$ or $H(E^\bullet)$ defined above. More precisely, we have:

\begin{proposition}
	For any $\phi, \psi \in G(E^\bullet)$,
	\begin{enumerate}
		\item $\phi_\ast V^\bullet = \psi_\ast V^\bullet$ if and only if $\psi = \phi \cdot f \text{ for some } f \in H(E^\bullet)$,
		\item $\phi_\ast C^\ast = \psi_\ast C^\ast$ if and only if $\psi = \phi \cdot f \text{ for some } f \in H'(E^\bullet)$,
	\end{enumerate}
	where the actions of $H(E^\bullet)$ and $H'(E^\bullet)$ on $G(E^\bullet)$ are given by Formula~\eqref{eq:action_on_symmetry_group} in Definition~\ref{def:symmetries_ellipsoid_flag}.
\end{proposition}

\begin{proof}
	(1) We have $V_\phi^n = \C^n = V_\psi^n$.
	Suppose that $V_\phi^k = V_\psi^k$ for all $1 \leq k \leq n$. Then $f_k = \bar{\phi}_k^{-1} \circ \bar{\psi}_k : V^k \rightarrow V^k$ is well-defined and preserves $V^{k-1}$. By Lemma~\ref{l:properties_bar_phi}, $f_k$ preserves also $E^k$, 
    hence $f_k \in H_k$. We then have, for any $1 \leq k < n$,
	\[ f_k^{-1} = \bar{\psi}_k^{-1} \circ \bar{\phi}_k = (\restrict{(\bar{\psi}_{k+1})}{V^k}{V^k_\psi} \circ \psi_k)^{-1} \circ (\restrict{(\bar{\phi}_{k+1})}{V^k}{V^k_\phi} \circ \phi_k) = \psi_k^{-1} \circ \restrict{(f_{k+1}^{-1})}{V^k}{V^k} \circ \phi_k, \]
	that is $\psi_k = \restrict{(f_{k+1}^{-1})}{V^k}{V^k} \circ \phi_k \circ f_k$.
	
	Conversely, if $\psi = \phi \cdot f$ for some $f=(f_1, \dots, f_n) \in H(E^\bullet)$, then for any $1 \leq k \leq n$, $\bar{\psi}_k = \bar{\phi}_k \circ f_k$ and hence
	\[ V_\psi^k = \bar{\psi}_{k+1}(V^k) = \bar{\phi}_{k+1} \circ f_{k+1}(V^k) = \bar{\phi}_{k+1}(V^k) = V_\phi^k. \]
	
	(2)
	If $\phi_\ast C^\ast = \psi_\ast C^\ast$, then $\phi_\ast V = \psi_\ast V$ by Proposition~\ref{p:push_forward_commute_with_q}, and hence $\psi = \phi \cdot f$ 
    for some $f \in H(E^\bullet)$. By definition of the action of $H(E^\bullet)$ on $G(E^\bullet)$ we have
	\[ u^\psi_k = \bar{\psi}_k(u_k) = \bar{\phi}_k \circ f_k(u_k). \]
	
	Finally, for any $1 \leq k \leq n$,
	\[ u^\psi_k = u^\phi_k \iff \bar{\phi}_k \circ f_k(u_k) = \bar{\phi}_k(u_k) \iff f_k(u_k) = u_k, \] 
	hence $\psi_\ast C^\ast = \phi_\ast C^\ast$ if and only if $f_k$ is the identity on $\Span(u_k) = (V^{k-1})^{\perp V^k}$ for any $k$, that is $f \in H'(E^\bullet)$.
\end{proof}

We want to prove now that the push-forward operation allows, starting from a flag $V^\bullet$, to reach all the flags lying in the same fiber (in the sense of Definition \ref{defn:fiber}) as $V^\bullet$,and similarly for a unitary basis formed by the columns of $C^\ast$. 

To do so, we first need the following result, stating that if a complex ellipsoid $E$ contains two codimension 1 ellipsoids $E_1$ and $E_2$ with same radii, then there exists a symmetry of $E$ mapping $E_1$ onto $E_2$.

\begin{lemma}
\label{l:moving_ellipsoids_inside_ellipsoid}
	Let $V_1, V_2$ be two codimension $1$ subspaces of some Hermitian space $V$. Consider $\alpha : V \rightarrow V$ a positive definite Hermitian form and $E_\alpha$ the corresponding ellipsoid in $V$. For $i=1,2$, let $\beta_i : V_i \rightarrow V_i$ be the positive Hermitian form defining the ellipsoid $E_{\beta_i} = E_\alpha \cap V_i$. Then there exists a unitary map $\varphi : V \rightarrow V$ such that
\begin{equation}
\begin{cases}
		\varphi(E_\alpha) = E_\alpha, \\
		\varphi(V_1) = V_2.
	\end{cases} 
\end{equation}
if and only if $\beta_1$ and $\beta_2$ have same eigenvalues.
\end{lemma}

\begin{proof} The ``only if" part is obvious. Let us now assume
that $\beta_1$ and $\beta_2$ have the same eigenvalues
$\lambda_1 \geq \cdots \geq \lambda_n$, where $n = \dim V - 1$. 
For $i=1,2$, denote by $(v_i^1, \dots, v_i^n)$ a unitary basis of eigenvectors of $\beta_i$. Fix a vector $\ell_i$ of length 1 n $L_i = (V_i)^\perp$ and set $\Bcal_i = (v_i^1, \dots, v_i^n, \ell_i)$. The matrix of $\alpha$ in the unitary basis $\Bcal_i$ has the form
	\[ A_i = \myblockmatrix{D_\lambda}{x_i}{x_i^\ast}{a_i} \]
	where $D_\lambda = \diag(\lambda_1, \dots, \lambda_n)$, $x_i = (x_i^1, \dots, x_i^n) \in \C^n$ and $a_i \in \R$.
	
	For every $X \in \R \setminus \{ \lambda_1, \dots, \lambda_n \}$, $D_\lambda - X I_{n}$ is invertible. Set $y_i = (D_\lambda - X I_{n})^{-1} x_i$. Then $y_i^\ast = x_i^\ast (D_\lambda - X I_{n})^{-1}$. It follows that
	\[ \myblockmatrix{I_n}{0}{y_i^\ast}{(a_i - X) - y_i^\ast x_i} \myblockmatrix{D_\lambda - X I_n}{x_i}{0}{1} = A_i - X I_{n+1}. \]
	Thus we obtain the following relation between the characteristic polynomial $P_\alpha(X)$ of $\alpha$ and the characteristic polynomial $P_\lambda(X) = (\lambda_1 - X) \cdots (\lambda_n - X)$ of $\beta_i$:
	\[ P_\alpha(X) = ((a_i - X) - y_i^\ast x_i) P_\lambda(X). \]
	But observe that
	\[ a_i = \tr A_i - \tr D_\lambda = \tr \alpha - (\lambda_1 + \cdots + \lambda_n) \]
	does not depend on $i \in \{ 1, 2 \}$. It follows that $y_1^\ast x_1 = y_2^\ast x_2$, that is
	\[ \sum_{j=1}^n \frac{\abs{x_1^j}^2}{\lambda_j - X} = \sum_{j=1}^n \frac{\abs{x_2^j}^2}{\lambda_j - X}. \]
	
	Write $\{ 1, \dots, n \} = I_1 \sqcup \cdots \sqcup I_k$ such that $j,j'$ are in a same $I_p$ if and only if $\lambda_j = \lambda_{j'} =: \lambda_{I_p}$. Define $\psi : V^1 \rightarrow V^2$ as follows. Identifying the coefficients in the above equality between rational functions, we obtain that for any $1 \leq p \leq k$,
	\[ r = \norm{x_i^{I_p}} = \sum_{j \in I_p} \abs{x_i^j}^2 \]
	does not depend on $i$. In other words $x_1^{I_p}$ and $x_2^{I_p}$ lie in sphere of same radius in $\C^{\abs{I_p}}$, hence there exists a matrix $g^p \in U(\abs{I_p})$ such that $x_2^{I_p} = g^p \cdot x_1^{I_p}$. More precisely, for any $j \in I_p$,
	\[ x_2^j = \sum_{j' \in I_p} g^p_{j,j'} x_1^{j'}. \]
	Then set for any $j \in I_p$,
	\[ \psi(v_1^j) = \sum_{j' \in I_p} g^p_{j',j} v_2^{j'} \]
	(note that $\psi(v_1^j)$ is an eigenvector of $\beta_2$ associated to the eigenvalue $\lambda_{I_p}$). Then $\psi$ is unitary and satisfies $\psi(x_1^1 v_1^1 + \cdots + x_1^n v_1^n) = x_2^1 v_2^1 + \cdots + x_2^n v_2^n$.
	
	Now extend $\psi$ to a map $\varphi : V \rightarrow V$ 
    by setting for all $v_1 \in V_1$ and $a \in \C$,
	\[ \varphi(v_1 \oplus a \ell_1) = \psi(v_1) \oplus a \ell_2. \]
	It is clear that $\varphi$ maps $V_1$ onto $V_2$. Let us check now that is preserves $E_\alpha$, or equivalently, that $\varphi \circ \alpha = \alpha \circ \varphi$.
	
	For any eigenvector $v_1^j \in V_1$, $j \in I_p$,
	\[ \alpha \circ \varphi(v_1^j) = \alpha \left( \sum_{j' \in I_p} g^p_{j',j} v_2^{j'} \right) = \sum_{j' \in I_p} g^p_{j',j} (\lambda_{I_p} v_2^{j'} \oplus \bar{x}_2^{j'} \ell_2) = \lambda_j \psi(v_1^j) \oplus c_j \ell_2  \]
	with $c_j = \sum_{j' \in I_p} g^p_{j',j} \bar{x}_2^{j'}$. But since the matrix $(g_{j,j'})$ is in $U(\abs{I_p})$, its inverse is $(\bar{g}_{j',j})$ and we have $c_j = \bar{x}_1^{j}$. Hence, $\alpha \circ \varphi(v_1^j)$ is equal to
	\[ \varphi \circ \alpha(v_1^j) = \varphi(\lambda_j v_1^j \oplus \bar{x}_1^j \ell_1) = \psi(\lambda_j v_1^j) \oplus \bar{x}_1^j \ell_2. \]
	Moreover,
	\[ \alpha \circ \varphi(\ell_1) = \alpha(\ell_2) = \sum_{j=1}^n x_2^j v_2^j \oplus a_2 \ell_2 = \varphi \left( \sum_{j=1}^{n} x_1^j v_1^j \oplus a_2 \ell_1 \right) = \varphi \circ \alpha(\ell_1) \]
	so $\phi : V \rightarrow V$ is the unitary transform given by the lemma.
\end{proof}

\begin{proposition}
	\label{p:push_forward_describes_full_fiber}
	Let $V_1^\bullet, V_2^\bullet \in \Fcal$ be two complete flags and $C_1^\ast, C_2^\ast \in U(n)$ bases for these flags, i.e., $Q(C_i^\ast) = V_i^\bullet$. Then the following three conditions are equivalent:
	\begin{enumerate}
		\item $\Gamma_\lambda(V_1^\bullet) = \Gamma_\lambda(V_2^\bullet)$,
		\item $V_2^\bullet = \phi_\ast V_1^\bullet \text{ for some } \phi \in G(E_1^\bullet)$,
		\item $C_2^\ast = \phi_\ast C_1^\ast \text{ for some } \phi \in G(E_1^\bullet)$.
	\end{enumerate}
\end{proposition}

\begin{proof}
	Set $E_i^\bullet = E_{D_\lambda} \cap V_i^\bullet$, so $\Gamma_\lambda(V_i^\bullet) = \Gamma(E_i^\bullet)$, and let $\alpha_\bullet = (\alpha_1, \dots, \alpha_n)$ (res. $\beta_\bullet = (\beta_1, \dots, \beta_n)$) be the family of unitary maps defining $E_1^\bullet$ (res. $E_2^\bullet$).
	
	(3) $\Rightarrow$ (2) is just a consequence of Proposition~\ref{p:push_forward_commute_with_q}.
	
	(2) $\Rightarrow$ (1): if $V_2^\bullet = \phi_\ast V_1^\bullet$, then by Lemma~\ref{l:properties_bar_phi},
	\[ E_{\beta_k} = E_2^k = E_{D_\lambda} \cap V_2^k = \bar{\phi}_k(E_1^k) = \bar{\phi}_k(E_{\alpha_k}) \]
	hence $\beta_k = \bar{\phi}_k \circ \alpha_k \circ \bar{\phi}_k^\ast$ and so $\alpha_k$ and $\beta_k$ have same eigenvalues.
	
	(1) $\Rightarrow$ (2): suppose $\Gamma_\lambda(V_1^\bullet) = \Gamma_\lambda(V_2^\bullet)$ and let us construct $\phi \in G(E^\bullet)$ recursively. Applying Lemma~\ref{l:moving_ellipsoids_inside_ellipsoid} to $V_1^{n-1}, V_2^{n-1} \subset \C^n$, we obtain a map $\phi_n : \C^n \rightarrow \C^n$ preserving $E_1^n$ and such that $\phi_n(V_1^{n-1}) = V_2^{n-1}$.
	
	Suppose constructed $\phi_n, \phi_{n-1}, \dots, \phi_{k+1}$ such that $V_2^{i-1} = \bar{\phi}_i(V_1^{i-1})$ for all $k < i \leq n$. Then 
	\[ W_2^{k-1} = \phi_{k+1}^{-1}(\phi_{k+2}^{-1}( \cdots \phi_n^{-1}(V_2^{k-1}) \cdots )) \subset V_1^k \]
	gives a ellipsoid $E_1^{k} \cap W_2^{k-1}$ in $V_1^k$ with same eigenvalues as $E_2^{k-1}$, and then same eigenvalues as $E_1^{k-1}$. Applying Lemma~\ref{l:moving_ellipsoids_inside_ellipsoid}, we obtain a map $\phi_k : V_1^k \rightarrow V_1^k$ preserving $E_1^k$ and such that $\phi_k(V_1^{k-1}) = W_2^{k-1}$, or equivalently,
	\[ \phi_n(\phi_{n-1}( \cdots \phi_k(V_1^{k-1}) \cdots )) = V_2^{k-1}. \]
	
	(1) $\Rightarrow$ (3): suppose $\Gamma_\lambda(V_1^\bullet) = \Gamma_\lambda(V_2^\bullet)$, then we just showed that there exists $\phi \in G(E_1^\bullet)$ such that $V_2^\bullet = \phi_\ast V_1^\bullet$. Set $C_2'^\ast = \phi_\ast C_1^\ast$. We have
	\[ Q(C_2'^\ast) = \phi_\ast Q(C_1^\ast) = \phi_\ast V_1^\bullet = V_2^\bullet = Q(C_2^\ast) \]
	hence there exists $T = \diag(\xi_1, \dots, \xi_n)$ such that $C_2^\ast = C_2'^\ast T$. For any $1 \leq i \leq n$, let $f_i : V_1^i \rightarrow V_1^i$ be the multiplication by the scalar $\xi_i \in U(1)$. Then $f = (f_1, \dots, f_n)$ is an element of $H(E_1^\bullet)$. Consider $\psi = \phi \cdot f \in G(E_1^\bullet)$. We have clearly $\bar{\psi}_i = \xi_i \bar{\phi}_i$, hence $\psi_\ast C_1^\ast = (\phi_\ast C_1^\ast) T = C_2'^\ast T = C_2^\ast$.
\end{proof}

	\subsection{The symmetry groupoid of ellipsoid flags}
    
Let us fix a value $\Lambda = q (V_0^\bullet) \in \mathbb{R}^N$
given by a complete flag $V_0^\bullet$,
and denote by $\mathcal{F}_\Lambda \ni V^\bullet$ the preimage of
$\Lambda$ under the map $Q$. For each complete flag $V^\bullet \in \Fcal$, let us denote by $G_{\lambda}(V^\bullet) = G(E^\bullet)$ the symmetry group of the ellipsoid flag $E^\bullet = E_{D_\lambda} \cap V^\bullet$, and let
\begin{equation}
\Gcal_\Lambda = {\Large\sqcup}_{V^\bullet \in \mathcal{F}_\Lambda} G_\lambda(V^\bullet) = \{ (V^\bullet, \phi) \mid V^\bullet \in \Fcal_\Lambda,\ \phi \in G_\lambda(V^\bullet) \} 
\end{equation}
be the disjoint union of those groups as $V^\bullet$ varies in $\Fcal_\Lambda$. Define the source and the target maps $s, t : \Gcal_\Lambda \rightarrow \Fcal_\Lambda$ by the following formula for $(V^\bullet, \phi)$ in $\Gcal_\Lambda$: 
\begin{equation}
\begin{cases}
	s(V^\bullet, \phi) = V^\bullet, \\
	t(V^\bullet, \phi) = \phi_\ast V^\bullet,
\end{cases} 
\end{equation}
where $\phi_\ast V^\bullet$ denotes the push-forward of flags defined by Formula \eqref{eqn:pushV}. The composition $\phi \circ \psi$ of two elements $(V^\bullet, \phi)$ and 
$(W^\bullet, \psi)$
such that $V = \psi_\ast W^\bullet$ is defined in an obvious way. For example, $(\phi \cdot \psi)_n = (\phi_n \circ \psi)_n$, $(\phi \cdot \psi)_{n-1} = ({\psi_n}|_{W_{n-1}})^{-1}
\circ \phi_{n-1} \circ ({\psi_n}|_{W_{n-1}}) \circ \psi_{n-1}.$
Similarly, there is a natural way to inverse the elements in
$\Gcal_\Lambda$. 
We have the following proposition, whose proof follows immediately
from the previous propositions: 

\begin{proposition}
	$\Gcal_\Lambda$ with the above natural maps and a natural smooth structure is a transitive Lie groupoid. 
\end{proposition}

We refer to~\cite{mackenzie1987lie} for the general theory of transitive Lie groupoids.
We will call  $\Gcal_\Lambda \rightrightarrows \Fcal_\Lambda$  it the \textbf{\textit{symmetry groupoid of ellipsoid flags}} for a given $\Lambda$. According to Proposition \ref{p:push_forward_describes_full_fiber}, the base manifold 
$\Fcal_\Lambda$ of this groupoid is diffeomorphic to
the quotient of the group $G_\lambda(V_0^\bullet)$ by the right
action of the group $H_\lambda(V_0^\bullet)$ given by
Formula \eqref{eq:action_on_symmetry_group}.

Denote by $U(n)^\ast_\Lambda$ the set of all $C^\ast \in U(n)$
such that  the flag $Q(C^\ast)$ generated by its columns belongs
to $\Fcal_\Lambda$. According to Proposition \ref{p:push_forward_describes_full_fiber}, we have that 
$U(n)^\ast_\Lambda$ is diffeomorphic to the quotient 
$S(E_0^\bullet) = G_\lambda(V_0^\bullet)/H'_\lambda(V_0^\bullet)$
of the group $G_\lambda(V_0^\bullet)$ by the right
action of the group $H'_\lambda(V_0^\bullet)$ given by
Formula \eqref{eq:action_on_symmetry_group}. 
Recall that this quotient space is called the reduced symmetry space of the ellipsoid flag 
$E_0^\bullet$, by Definition \ref{def:symmetries_ellipsoid_flag}. Moreover, the map
$q$ restricted to $U(n)^\ast_\Lambda$ is a submersion from
$U(n)^\ast_\Lambda$ to $\Fcal_\Lambda$. 

There is a natural smooth action of the Lie groupoid $\Gcal_\Lambda \rightrightarrows \Fcal_\Lambda$ on the submersion
$Q: U(n)^\ast_\Lambda \rightarrow \Fcal_\Lambda$ defined as follows: if $\phi \in G_\lambda(V^\bullet)$ and $C^\ast 
\in U(n)^\ast_\Lambda$ such that $Q(C) = V^\bullet$ then
\begin{equation}
\phi \cdot C^* = \phi_\ast C^*.
\end{equation}
Moreover, this action is transitive on  $U(n)^\ast_\Lambda$.

The above discussions lead to the following result on the smoothness of singular fibers of the Gelfand-Cetlin system:

\begin{theorem}
	\label{t:GC_fibers_are_manifolds}
    For any value $\Lambda$ of the Gelfand-Cetlin map $F_\lambda$,
the fiber $F_\lambda^{-1}(\Lambda)$ of the Gelfand--Cetlin system on $\orbit{\lambda}$ is an embedded smooth manifold diffeomorphic to the quotient manifold
\begin{equation}U(n)^\ast_\Lambda / U_\lambda 
\end{equation}
	where $U_\lambda$ is the group of unitary matrices commuting with $D_\lambda = \diag(\lambda_1, \dots, \lambda_n)$ and its free (right) action on $U(n)^\ast_\Lambda$ is defined by
\begin{equation}
    C^* \cdot P = P^* C^*
\end{equation}    
	for all $P \in U_\lambda$ and $C^* \in U(n)^\ast_\Lambda$.
\end{theorem}

Remark that, instead of looking at $U(n)^\ast_\Lambda$, we can also look at its inversion $U(n)_\Lambda$ in $U(n)$, and write the above quotient as $U(n)_\Lambda / U_\lambda $ with the right action
$C \cdot P = CP$.
\begin{proof} The proof is straightforward: two matrices
$C_1, C_2 \in U(n)_\Lambda$ give rise to the same element
in the fiber $F^{-1}(\Lambda) \in \orbit{\lambda}$ of the 
Gelfand-Cetlin system if and only if $C_1 D_\lambda C_1^*
= C_2 D_\lambda C_2^*$, which means that $P D_\lambda = D_\lambda P$, where $P = C_2^*C_1$, that is, $P \in U_\lambda$.
(Recall that $C^* = C^{-1}$ for unitary matrices $C$).
\end{proof}


\section{Dimensions, decomposition and examples of the fibers}

In this section, based on the results of the previous section,
we compute the dimensions of the fibers of 
Gelfand--Cetlin systems. We also show a topological decomposition
of these fibers into direct products, and describe explicitly 
some examples.

As before, we denote by $\lambda = (\lambda_1 \geq \cdots \geq \lambda_n > 0)$ a fixed positive spectrum, by $2N$ the dimension of the coadjoint orbit $\orbit{\lambda}$ and by $F_\lambda : \orbit{\lambda} \rightarrow \R^N$ the momentum map of the Gelfand--Cetlin system on $\orbit{\lambda}$. We also consider $\Fcal$ the set of complete flags in $\C^n$ and $\Gamma_\lambda : \Fcal \rightarrow \R^N$ the map that associates to a flag $V^\bullet$ the eigenvalues $\Gamma_\lambda(V^\bullet)$ of the ellipsoid flag $E_{D_\lambda} \cap V^\bullet$. We denote by $p : U(n) \rightarrow \orbit{\lambda}$ and $q : U(n) \rightarrow \Fcal$ the natural projections. Together with the involution $t : U(n) \rightarrow U(n)$, $t(C) = C^\ast$, these maps define a commutative diagram given in Figure~\ref{fig:cd_change_of_pov}.

In each case we fix a value $\Lambda = (\lambda_{i,j})_{1 \leq i \leq j < n}$ in $\R^N$, and we fix a flag $V^\bullet$ in $\Gamma_\lambda^{-1}(\Lambda)$. We denote by $\alpha_\bullet = (\alpha_1, \dots, \alpha_n)$ the family of Hermitian operators $\alpha_k:V^k \to V^k$ defining the ellipsoid flag $E^\bullet = E_{D_\lambda} \cap V^\bullet$. By our assumptions, for each $1 \leq k \leq n$, $\alpha_k$ has eigenvalues $\lambda_{1,k} \geq \lambda_{2,k} \geq \cdots \geq \lambda_{k,k}$.

\subsection{Dimensions of the fibers}

Using the notations and the results of the previous section, we have the following formula:
\begin{equation}
\dim F^{-1}(\Lambda) = \dim 
S(E^\bullet) - \dim U_\lambda = \dim G_\lambda(V^\bullet) - \dim H'_\lambda(V^\bullet) - \dim U_\lambda, 
\end{equation}
which can also be written
as
\begin{equation}
\dim F^{-1}(\Lambda) =  \sum_{k=1}^n (\dim G_k  - \dim H'_k) - \dim U_\lambda, 
\end{equation}
where $G_k$ is the group of unitary transformations of $V^k$ preserving $E^k = E_{D_\lambda} \cap V^k$; $H_k$ is the group of unitary transformations of $V^k$ preserving both $E^k$ and $E^{k-1}$
and such that their restriction to $(V^{k-1})^{\perp V^{k}}$ is the identity; $U_\lambda$ is the subgroup of unitary matrices which commute
with $\diag(\lambda_1,\hdots,\lambda_n)$.

\subsection{Regular fibers on generic coadjoint orbits}
In the regular case on a generic coadjoint orbit, i.e. where all the inequalities
in the Gelfand-Cetlin triangle are strict, all the eigenvalues of each
ellipsoid in the ellipsoid flags are distinct and are different from
those of the subsequent ellipsoid (i.e. the one of one dimension smaller), for each $k=1,\hdots,n$ we have that $H'_k$ is trivial
and $G_k = U(W_{1,k}) \oplus \cdots \oplus U(W_{k,k})$ where each
$W_{i,k}$ is a 1-dimensional eigenspace of $\alpha_k$, so we have
$G_k \cong U(1)^k \cong \mathbb{T}^k$. Similarly, $U_\lambda$ consists
of only diagonal unitary matrices, so $U_\lambda \cong \mathbb{T}^n$.
Thus, in the regular case on a generic coadjoint orbit, we have 
\begin{equation}
\dim F^{-1}(\Lambda) =  (\sum_{k=1}^n k) - n = \dfrac{n(n+1)}{2} - n 
= \dfrac{n(n-1)}{2} = N, 
\end{equation}
which is exactly half the dimension $n(n-1)$ of a generic coadjoint orbit of $U(n)$. This fact is not surprising, since we know that
connected regular fibers of an integrable Hamiltonian systems are 
Lagrangian tori whose dimension is equal to half the dimension 
of the symplectic manifold. In our regular generic case, 
$S(E^\bullet) \cong U(n)_\Lambda$ is  a torus of dimension $n(n+1)/2$
on which a torus $U(n)_\lambda \cong \mathbb{T}^n$ of dimension $n$
acts freely, and the quotient space $F^{-1}(\Lambda)$ is a torus
of dimension $n(n+1)/2 - n = n(n-1)/2 = N$.

\subsection{Elliptic nondegenerate singular fibers}

In the case when all the horizontal inequalities in the Gelfand-Cetlin triangle are strict, but there are some diagonal inequalities
(of the types $\lambda_{i,k} = \lambda_{i-1,k-1}$ or $\lambda_{i,k} = \lambda_{i,k-1}$), we still have $G_k \cong \mathbb{T}^k$ for each $k$
and $U_\lambda \cong \mathbb{T}^n$, but now some of the groups
$H'_k$ are non-trivial tori. For example, if $\lambda_{i,k} = \lambda_{i-1,k-1}$ then it means that the corresponding eigenspaces
of $\alpha_{k-1}$ and $\alpha_k$ coincide: $W_{i,k} = W_{i-1,k-1}$
which implies that $H'_{k}$ contains $U(W_{i,k})$. In general, if there are $s(k) \geq 0$ equalities between the numbers $\lambda_{1,k},\hdots,
\lambda_{k,k}$ and the numbers $\lambda_{1,k-1},\hdots,
\lambda_{k-1,k-1}$ then the group $H'_k$ is isomorphic to $\mathbb{T}^{s(k)}$. Thus, in this case we have
\begin{equation}
S(E^\bullet) \cong \mathbb{T}^{\frac{n(n+1)}{2} - \sum_k s(k)}
\end{equation}
and the fiber
\begin{equation}
F^{-1}(\Lambda) \cong \mathbb{T}^{N - \sum_k s(k)}
\end{equation}
is a torus of dimension $N - \sum_k s(k)$, where $N = \dfrac{n(n-1)}{2}$ is the dimension of  regular fibers and $\sum_k s(k)$ is the total number of diagonal equalities in the Gelfand-Cetlin triangle. 

Remark that, in this case, the momentum map $F$ is smooth at 
$F^{-1}(\Lambda)$, and $F^{-1}(\Lambda)$ is an elliptic singularity
of corank $\sum_k s(k)$, so the fact that it is a torus of dimension
$N - \sum_k s(k)$ fits well with the general theory of nondegenerate singularities of integrable Hamiltonian systems (see \cite{zung1996symplectic}).

\subsection{Spherical singularity of a generic orbit of $\uu(3)^\ast$} 
\label{subsection:spherical}

Take $n = 3$ and consider a generic coadjoint orbit of dimension $2N = 6$. Choose $\Lambda = (\lambda_{1,1}, \lambda_{2,1}, \lambda_{2,2})$ such that $\lambda_{2,1} = \lambda_{2,2}$. The inequalities in the Gelfand--Cetlin diagram implies that we have $\lambda_{1,1} = \lambda_{2,1} = c_{2,2} = \lambda_2$.

\subsubsection*{Computation of $G(E^\bullet)$}
We simply apply Proposition~\ref{p:structure_G_k} to obtain:
\begin{itemize}
	\item $G_1 = U(V^1)$,
	\item $G_2 = U(V^2)$,
	\item $G_3 = U(W_1) \oplus U(W_2) \oplus U(W_3)$ where each $W_i$ has dimension 1.
\end{itemize}

\subsubsection*{Computation of $H'(E^\bullet)$} For $0 \leq k < 3$, let $L_k$ be the orthogonal complement of $V^k$ in $V^{k+1}$. Set $A_k = \proj_{V^k}(\alpha_{k+1}(L_k))$ and denote by $A'_k$ its orthogonal complement in $V^k$. By Proposition~\ref{p:structure_H_k+1}, $H'_{k+1}$ is the set of all transforms
\[ \phi_k \oplus \mathrm{id}_{L_k} \]
with $\phi_k \in U(V^k)$ satisfying $\phi_k(v) = v$ for any $v \in A_k$. We then compute the dimension of $A_k$ to determine whether the latter condition on $\phi_k$ is trivial or not.
\begin{itemize}
	\item By definition, $H_1 = G_1 = \{ \xi . \mathrm{id}_{V^1} \}$ and $H'_1 = \{ \mathrm{id}_{V^1} \}$.
	
	\item Fix $\ell_1$ a unit vector in $L_1$ and denote $\alpha_2(\ell_1) = w_1 \oplus a_2 \ell_1$. Then by Lemma~\ref{l:alpha_k+1_depends_on_alpha_k}, for all $v_1 \in V^1$,
	\[ \alpha_2(v_1) = \alpha_1(v_1) \oplus \dotprod{v_1}{w_1} \ell_1. \]
	But remark that both $\alpha_1$ and $\alpha_2$ are simply the scalar multiplication by $\lambda_2$, so finally we have $\dotprod{v_1}{w_1} = 0$ for all $v_1 \in V^1$, that is $w_1 = 0$. It follows that $A_1 = \Span(w_1) = \{ 0 \}$, and then
	\[ H'_2 = \{ \phi_1 \oplus \mathrm{id}_{L_1} \mid \phi_1 \in U(V^1) \}. \]
	
	\item Since $A_2$ has dimension at most $1$, its complement $A'_2$ has dimension $1$ or $2$. But recall that $H_3$ is diffeomorpic to $U(A'_2) \times U(1)$, so
	\[ (\dim_\C A'_2)^2 + 1 = \dim_\R H_3 \leq \dim_\R G_3 = 3. \]
	It follows that $\dim_\C A'_2 = 1 = \dim_\C A$. Hence $H'_3$ rewrites as the subgroup
	\[ H'_3 = \{ \xi_3 . \mathrm{id}_{A'_2} \oplus \mathrm{id}_{A_2} \oplus \mathrm{id}_{L_3} \mid \xi_3 \in U(1) \}. \]
\end{itemize}

\subsubsection*{Quotient space $S(E^\bullet) = G(E^\bullet) / H'(E^\bullet)$}

We are going to write all the transformations in $G_1, G_2, G_3, H'_2, H'_3$ in the same basis $C = (u_1, u_2, u_3)$, which by definition satisfies $V^k = \Span(u_1, \dots, u_k)$ for all $1 \leq k \leq 3$.

First let us remark that $W_2 \subset V^2$ and is actually equal to $A'_2$. Indeed suppose $v_3 \in W_2$ and write
\[ v_3 = v_2 \oplus c \ell_2, \quad v_2 \in V^2, \ c \in \C \]
where $\ell_2$ is some fixed unit vector in $L_2$. Recall that by Lemma~\ref{l:alpha_k+1_depends_on_alpha_k}, if we write $\alpha_3(\ell_2) = w_2 \oplus a_3 \ell_2$ with $w_2 \in V^2$ and $a_3 \in \R$, then $\alpha_3(v_2) = \lambda_2 v_2 \oplus \dotprod{v_2}{w_2} \ell_2$. Since $v_3 \in W_2$, it satisfies $\alpha_3(v_3) = \lambda_2 v_3$, which rewrites as
\[ \begin{cases}
	\lambda_2 v_2 + c w_2 = \lambda_2 v_2, \\
	\dotprod{v_2}{w_2} + c a_3 = \lambda_2 c.
\end{cases} \]
We saw previously that $A_2 = \proj_{V^2} (\alpha_3(L_2)) = \Span(w_2)$ has dimension $1$, that is $w_2 \neq 0$. It follows that $c = 0$, hence $v_3 = v_2 \in V^2$. Moreover the above system gives also $\dotprod{v_3}{w_2} = \dotprod{v_2}{w_2} = 0$, that is $v_3$ is orthogonal to $A_2$. Since the dimensions agree, we obtain that $W_2 = (A_2)^{\perp V^2} = A'_2$.

Then, note that without loss of generality we can assume that $V^1 = A_2$. Indeed, take $\phi_2 \in U(V^2)$ mapping $V^1$ onto $A_2$, and set $\phi = (\mathrm{id}_{V^1}, \phi_2, \mathrm{id}_{V^3}) \in G(E^\bullet)$. The push-forward $\phi_\ast V^\bullet$ satisfies $V^2_\phi = V^2$ and $V^3_\phi = V^3$, hence only the $1$-dimensional ellipsoid $E^1$ has been modified. In particular the spaces $L_k, A_k, A'_k$ remain unchanged for $k \geq 2$. But according to the Gelfand--Cetlin diagram, the eigenvalue of $E_\phi^1$ is necessarily $\lambda$, hence $C' = \phi_\ast C$ lies in the same fiber of $\Gamma_\lambda$ as $C$.

\begin{figure}
	\centering
	\includegraphics{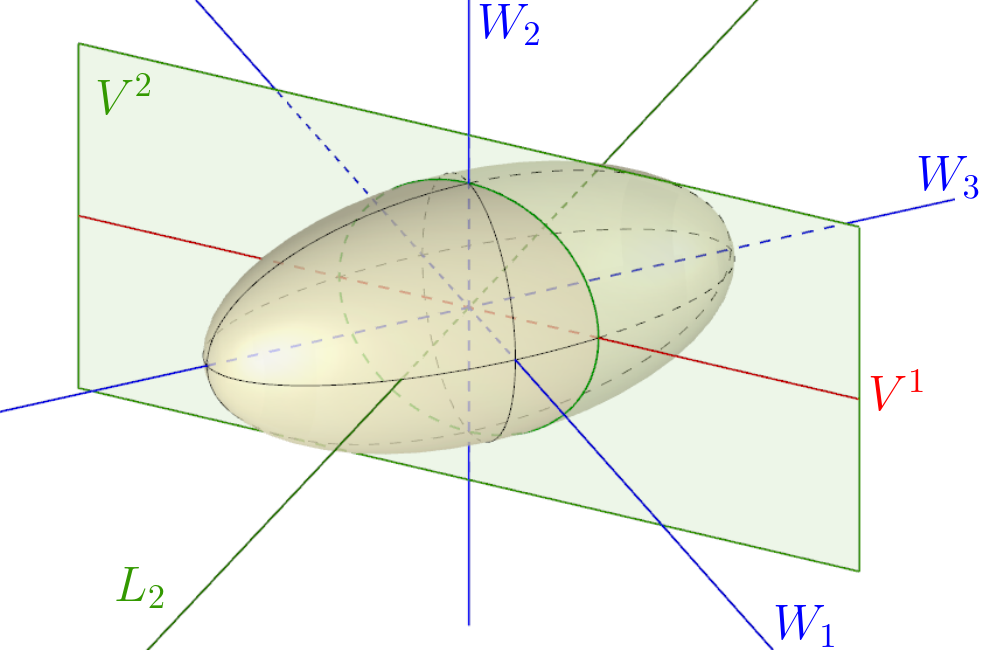}
	\caption{A complete flag in $\C^3$ corresponding to a matrix in the spherical singular fiber on $\uu(3)^\ast$}
	\label{f:spherical_singularity_flag}
\end{figure}
Figure~\ref{f:spherical_singularity_flag} provides a geometric interpretation of the above two facts. Consider $E^3$ a generic ellipsoid in $\C^3$ with semi-principal axes $v_1, v_2, v_3$ of lengths $a_1 < a_2 < a_3$ (with $a_i = 1/\sqrt{\lambda_i}$). If $E^2 = E^3 \cap V^2$ is a circle of radius $a_2$ contained in $E^3$, then $v_2$ must be an axis of $E^2$, giving the condition $W_2 = \Span(v_2) \subset V^2$. Moreover, any ellipsoid $E^1 = E^2 \cap V^1$ has radius $a_2$ too, so the value $c = \Gamma_\lambda(E^\bullet)$ does not depend on the choice of $V^1$.

Finally, if we denote by $J_1$ the identity on $A_2 = V_1 = \Span(u_1)$, by $J_2$ the identity on $W_2 = A'_2 = L_1 = \Span(u_2)$ and by $J_3$ the identity on $L_2 = \Span(u_3)$, we have:
\begin{itemize}
	\item $G_1 = \{ \zeta_1 J_1 \}$,
	\item $G_2 = U(V^2) = U(\C u_1 \oplus \C u_2)$,
	\item $G_3 = \{ \zeta_3 J_2 \oplus \psi_3 \mid \zeta_3 \in U(1),\ \psi_3 \in T \}$,
	\item $H'_2 = \{ \xi_2 J_1 \oplus J_2 \mid \xi_2 \in U(1) \}$,
	\item $H'_3 = \{ J_1 \oplus \xi_3 J_2 \oplus J_3 \mid \xi_3 \in U(1) \}$.
\end{itemize}
where $T$ is a subgroup of $U(\C u_1 \oplus \C u_3)$ diffeomorphic to $\T^2$. Define a map $\tilde{\varphi} : G(E^\bullet) \rightarrow G_2 \times T$ by setting for all $\phi = (\zeta_1 J_1, \phi_2, \zeta_3 J_2 \oplus \psi_3) \in G(E^\bullet)$,
\[ \tilde{\varphi}(\phi) = ((J_1 \oplus \zeta_3 J_2) \circ \phi_2 \circ (\zeta_1 J_1 \oplus J_2), \psi_3). \]
For any $f = (J_1, \xi_2 J_1 \oplus J_2, J_1 \oplus \xi_3 \oplus J_3) \in H'(E^\bullet)$,
\[ \phi \cdot f = (\xi_2^\ast \zeta_1 J_1, (J_1 \oplus \xi_3^\ast J_2) \circ \phi_2 \circ (\xi_2 J_1 \oplus J_2), \xi_3 \zeta_3 J_2 \oplus \psi_3). \]
The map $\tilde{\varphi}$ is $H'(E^\bullet)$-invariant and induces a diffeomorphism
\[ \varphi : S(E^\bullet) \overset{\sim}{\longrightarrow} G_2 \times T \isomorphic U(2) \times \T^2  \]

\subsubsection*{Quotient space $S(E^\bullet)/U_\lambda$}

Note that modulo $H'(E^\bullet)$, any $\phi \in G(E^\bullet)$ can be written
\[ \phi = (J_1, \phi_2, J_2 \oplus \psi_3), \quad \phi_2 \in G_2 = U(V^2),\ \psi_3 \in T \]
in a unique way, providing an explicit inverse $\varphi^{-1} : G_2 \times T \rightarrow S(E^\bullet)$. For such a $\phi$, the push-forward $\phi_\ast C = (u_1^\phi, u_2^\phi, u_3^\phi)$ is given by
\[ \begin{cases}
	u_3^\phi = \psi_3(u_3), \\
	u_2^\phi = (J_2 \oplus \psi_3)(\phi_2(u_2)), \\
	u_1^\phi = (J_2 \oplus \psi_3)(\phi_2(u_1)).
\end{cases} \]
This can be written in the more convenient form:
\[ \Forall 1 \leq k \leq 3, \quad u_k^\phi = (J_2 \oplus \psi_3) \circ (\phi_2 \oplus J_3)(u_k). \]

Let us now describe the action of $U_\lambda$ on $S(E^\bullet)$ explicitly. For any $P = \diag(\xi_1^\ast, \xi_2^\ast, \xi_3^\ast) \in U_\lambda$, the matrix $C' = P^\ast (\phi_\ast C)$ has columns $(u'_1, u'_2, u'_3)$ given by the relation
\[ \Forall 1 \leq k \leq 3, \quad u'_k = (\xi_1 I_1 \oplus \xi_2 I_2 \oplus \xi_3 I_3)(u_k^\phi), \]
where $I_j$ denote the identity map on $W_j = \C e_j$. Note that $J_2 = I_2$ and recall that $\psi_3 = \zeta_1 I_1 \oplus \zeta_3 I_3$ for some $\zeta_1, \zeta_2 \in U(1)$. It follows that
\[ (\xi_1 I_1 \oplus \xi_2 I_2 \oplus \xi_3 I_3) \circ (J_2 \oplus \psi_3) \circ (\phi_2 \oplus J_3) = (J_2 \oplus \psi'_3) \circ (\phi_2' \oplus J_3) \]
where $\psi_3' = (\xi_1 I_1 \oplus \xi_3 I_3) \circ \psi_3 = \xi_1 \zeta_1 I_1 \oplus \xi_3 \zeta_3 I_3 \in T$ and
\[ \phi'_2 = (J_1 \oplus \xi_2 J_2) \circ \phi_2 \in G_2. \]
Hence $P^\ast (\phi_\ast C) = \phi'_\ast C$ with $\phi' = (J_1, \phi'_2, J_2 \oplus \psi'_3)$.

Consider the smooth map
\[ \begin{array}{r|rcl}
	\tilde{\Psi} & S(E^\bullet) & \longrightarrow & \sphere{3} \\
		& (J_1, \phi_2, J_2 \oplus \psi_3) & \longmapsto & \phi_2^{-1}(u_1)
\end{array} \]
with values in the $3$-sphere $\sphere{3}$. It is clearly onto. With the above notation, note that
\[ (\phi'_2)^{-1}(u_1) = \phi_2^{-1} \circ (J_1 \oplus \xi_2^\ast)(u_1) = \phi_2^{-1}(u_1), \]
so $\tilde{\Psi}$ is constant along the orbits of the action of $U_\lambda$ on $S(E^\bullet)$. Moreover, suppose $(\phi'_2)^{-1}(u_1) = \phi_2^{-1}(u_1) = v_1$ for two maps $\phi_2, \phi'_2 \in G_2 = U(V^2)$. Let $v_2 = \phi_2^{-1}(u_2)$ and $v'_2 = (\phi'_2)^{-1}(u_2)$. Since $\phi_2 \in U(V^2)$, $(v_1, v_2)$ and $(v_1, v'_2)$ are two unitary bases of $V^2$. It follows that there exists $\xi_2^\ast \in U(1)$ such that $v'_2 = \xi_2 v_2$, hence $(\phi'_2)^{-1} = \phi_2^{-1} \circ (J_1 \oplus \xi_2^\ast J_2)$. Finally, the level-sets of $\tilde{\Psi}$ are exactly the orbits of $U_\lambda$ since for any $\psi_3 = \zeta_1 I_1 \oplus \zeta_3 I_3$ and $\psi'_3 = \zeta'_1 I_1 \oplus \zeta'_3 I_3$ we have $\psi'_3 = (\xi_1^\ast I_1 \oplus \xi_3^\ast I_3) \circ \psi_3$. Hence $\tilde{\Psi}$ induces a diffeomorphism $\Psi : S(E^\bullet)/ U_\lambda \rightarrow \sphere{3}$ and by Theorem~\ref{t:GC_fibers_are_manifolds} we recover the following result of the  thesis of Alamiddine~\cite{alamiddine2009GelfandCeitlin}:
\begin{proposition}
	Let $F_\lambda : \orbit{\lambda} \rightarrow \R^3$ be the Gelfand--Cetlin system on the generic coadjoint orbit $\mathfrak{u}(3)^\ast$ corresponding to eigenvalues $\lambda_1 > \lambda_2 > \lambda_3$.
Then the singular fiber $F_\lambda^{-1}(\lambda_2, \lambda_2, \lambda_2)$ is an embedded submanifold of $\orbit{\lambda}$ diffeomorphic to the $3$-dimensional spĥere $\sphere{3}$.
\end{proposition}

\begin{remark} An important contribution of \cite{alamiddine2009GelfandCeitlin} is the symplectic normal form proved for a neighborhood of these fibers: The symplectic model is that of the geodesic flow of the canonical (symmetric) metric on $S^3$.
\end{remark}

\subsection{Diamond singularities}

Consider now the case when there are some horizontal inequalities
in the Gelfand-Cetlin triangle of eigenvalues. Say $\lambda_{i,k} = \lambda_{i+1,k}$. In this case, the corresponding fiber $F^{-1}(\Lambda)$ will be called a \textit{\textbf{degenerate singularity}} of the Gelfand-Cetlin system. Notice that the components 
$F_{i,k}$ and $F_{i+1,k}$ of the Gelfand-Cetlin momentum map
(whose values at $F^{-1}(\Lambda)$ are equal, $\lambda_{i,k} = \lambda_{i+1,k}$)
are non-smooth at $F^{-1}(\Lambda)$, though their sum and their product
are still smooth functions there, but if we change $F_{i,k}$ and $F_{i+1,k}$ by $F_{i,k} + F_{i+1,k}$ and $F_{i,k}F_{i+1,k}$ in the
momentum map to make it smooth then the new momentum map will have a degenerate singularity at $F^{-1}(\Lambda)$ (see \cite{eliasson1990normal,zung1996symplectic,MirandaZung_NF2004} for the notion of nondegenerate singularities of smooth integrable systems). That's why we call 
$F^{-1}(\Lambda)$ a degenerate singularity.

Assume now, for example, that we have
\begin{equation}
\hdots > \lambda_{i-1,k} > \lambda_{i,k} =
\lambda_{i+1,k} = \hdots = \lambda_{i+l-1,k} > \lambda_{i+l,k} > \hdots,
\end{equation}
i.e. we have exactly $l$ equal eigenvalues $\lambda_{i,k} =
\lambda_{i+1,k} = \hdots = \lambda_{i+l-1,k} = \gamma$ on line $k$
of the Gelfand-Cetlin triangle, $l \geq 2$. Then automatically we must have a 
diamond of equalities as shown on Figure \ref{f:diamond} (assuming that our coadjoint orbit is generic, i.e. on line $n$ of the Gelfand-Cetlin triangle there is no equality).

\begin{figure}[!ht]
	\centering
	\begin{tikzpicture}[y=-1cm, x=0.8cm]
	\newcommand{\nodet}[2]{\node at (#1) {$#2$};}
	\newcommand{\nodem}[3]{\node at (#1) {$\lambda_{#3, #2}$};}
	\newcommand{\nodef}[1]{\node[rotate=-45] at (#1) {$=$};}
	\newcommand{\noder}[1]{\node[rotate=45] at (#1) {$=$};}
	\newcommand{\nodeh}[1]{\node at (#1) {$=$};}
    
    \nodem{5,-1}{k+l-1}{i+l-1}

	\noder{4.5,-0.5}  \nodef{5.5,-0.5}        	
 \nodem{3,1}{k+1}{i+1} \nodeh{4,1}  \nodet{5,1}{\cdots} \nodeh{6,1} \nodem{7.2,1}{n-1}{n-1}
\node[rotate=45] at (3.5, 0.5) {$\cdots$}; \node[rotate=-45] at (6.5, 0.5) {$\cdots$}; 
    \nodef{4.5,0.5} \noder{5.5,0.5} 
	\nodem{2,2}{k}{i} \nodeh{3,2} \nodem{4,2}{k}{i+1} \nodeh{5,2} \nodet{6,2}{\cdots} \nodeh{7,2} \nodem{8,2}{k}{i+l-1}
 \noder{2.5,1.5} \nodef{3.5,1.5} \noder{6.5,1.5} \nodef{7.5,1.5} 
	\node[rotate=-45] at (3, 3) {$\cdots$}; \node[rotate=45] at (7, 3) {$\cdots$};
	\nodef{3.5,3.5} \noder{4.5,3.5} \nodef{5.5,3.5} \noder{6.5,3.5}
	\nodem{3.8,4}{k-l+2}{i} \nodeh{5,4} \nodem{6.3,4}{k-l+2}{i+1}
	\nodef{4.5,4.5}  \noder{5.5,4.5}
	\nodem{5,5}{k-l+1}{i}
	\end{tikzpicture}
	\caption{A diamond of equalities}
	\label{f:diamond}
\end{figure}

Assume that there are no other equalities except the ones in the diamond on Figure
\ref{f:diamond}. Then we will say that $\Lambda$ is a diamond singular value  and $F^{-1}(\Lambda)$ a \textit{\textbf{diamond singularity}} of the Gelfand-Cetlin system.
In this case, the groups $G_j$ and $H'_j$ ($1 \leq j \leq n$) are as follows (with the above indices
$i, k, l$ such that $i+l-1 \leq k$ and $k+l-1 \leq n$):

\begin{itemize}
\item $G_j \cong \mathbb{T}^j$ for $j \leq k-l + 1$ and for $j \geq k+l - 1$;
$G_j \cong U(l - |k-j|) \times \mathbb{T}^{k-l+ |k-j|}$ for $k-l+1 < j < k+l-1$.
(The number $l - |k-j|$, when it's positive, is the number of eigenvalues on line
$j$ in the Gelfand-Cetlin triangle which are equal to the given value $\gamma$).
\item $H'_j$ is trivial for $j \leq k-l+1$ and for $j \geq k+l$;
 $H'_j \cong U(l - 1/2 - |k-j+1/2|)$ for 
 $k-l+1 < j < k+l$. For example, $H'_{k+l-1} \cong H'_{k-l+2} \cong U(1) \cong \mathbb{T}^1$, and $H'_{k} \cong H'_{k+1} \cong U(l-1)$.
\end{itemize}

It follows from the above computations that in this case we have
\begin{equation}
S(E^\bullet) \approx U(l) \times \mathbb{T}^{N+n-l^2}
\approx SU(l) \times \mathbb{T}^{N+n-l^2+1}
\end{equation}
and the diamond singular fiber
\begin{equation}
F^{-1}(\Lambda) \approx SU(l) \times \mathbb{T}^{N-l^2+1}
\end{equation}
is of dimension $N$, equal to the dimension of a regular fiber.

The spherical singularity studied in detail in Subsection \ref{subsection:spherical} is a special case of diamond
singularities, where we have $n=3, i=1, k=2, l =2, N= 3$, 
$N - l^2 + 1 = 0$, and the fiber is diffeomorphic to 
$SU(2) \approx S^3$. 

The above computations can be generalized to the case of multiple
diamonds, when the set of equalities in the Gelfand-Cetlin triangle form exactly $s \geq 1$ diamonds of sizes $l_1, \hdots, l_s$ respectively which do not intersect each other. We may call it a \textit{\textbf{multiple-diamond singularity}}.
A multiple-diamond fiber will be diffeomorphic to
\begin{equation}
SU(l_1) \times \hdots \times SU(l_s) \times \mathbb{T}^{N- \sum (l_i^2+1)}
\end{equation}
which is still of the same dimension $N$ as the regular fibers.

\subsection{General degenerate singularities}

For a general degenerate singularity, instead of diamonds 
of equalities in the Gelfand-Cetlin triangle, we have 
parallelograms of equalities which may overlap each other
in a vertical fashion to create connected chains of equalities. (The parallelograms do not overlap each other
horizontally, in the sense that if we take a connected chain
of overlapping parallelograms of equalities, then the intersection of that chain with each line of the Gelfand-Cetlin triangle
is connected if not empty). The topology of the singular fibers
can be described by these chains of parallelograms: each point
in the Gelfand-Cetlin triangle which does not belong to any chain and is not in the uppermost line corresponds to a factor
$\mathbb{T}^1$, each chain corresponds to some complicated factor, and $F^{-1}(\Lambda) $ is diffeomorphic to the direct
product of all these factors. 

In the case when there is just one parallelogram of equal eigenvalues of size $a \times b$ 
(whose vertices in the Gelfand-Cetlin triangle have indices, say, $(p,q)$, $(p,q+a-1)$, $(p+b-1,q+b-1)$ and $(p+b-1,q+a+b-2)$
for some $p,q$), and $\orbit{\lambda}$ is a generic coadjoint orbit (i.e., the inequalities in the top line of the 
Gelfand-Cetlin triangle are strict), 
then by computations of the groups 
$G_j$ and $H'_j$ similar to the previous subsection, we get the formula
\begin{equation}
\dim F^{-1}(\Lambda) = N - \min(a,b). |a-b|
\end{equation}
In particular, if $a \neq b$ then $\dim F^{-1}(\Lambda) < N$,
and if $a=b$then we have a diamond singularities and
$\dim F^{-1}(\Lambda) = N$. 

In general, we always have
\begin{equation}
\dim F^{-1}(\Lambda) \leq N
\end{equation}
for any fiber $\dim F^{-1}(\Lambda)$. Instead of proving this directly via combinatorial formulas in the general case, one
can deduce it from Corollary \ref{cor:isotropic} proved in the next section, which says that all the fibers of the Gelfand-Cetlin system are isotropic.

We observe that \textit{\textbf{symmetrically overlapping diamonds}} also give rise to singular fibers of the same dimension as regular fibers. For example, consider the case
with $\lambda_{1,1} = \lambda_{1,2} = \lambda_{2,2} 
= \lambda_{1,3} = \lambda_{2,3} = \lambda_{3,3}  
= \lambda_{2,4} = \lambda_{3,4} = \lambda_{2,5}  
= \lambda_{2,5} = \lambda_{3,5} = \lambda_{4,5} 
= \lambda_{3,6} = \lambda_{4,6} = \lambda_{4,7} $
and no other equalities. This is a case of of 2 
overlapping diamonds of size $3 \times 3$. In this case, we have
$\dim G_1 = 1, \dim H'_1 = 0$ (as always),
$\dim G_2 = 4, \dim H'_2 = 1$, 
$\dim G_3 = 9, \dim H'_3 = 4$,
$\dim G_4 = 4 + 2, \dim H'_4 = 4$,
$\dim G_5 = 9 + 2, \dim H'_5 = 4$,
$\dim G_6 = 4 + 4, \dim H'_6 = 4$,
$\dim G_7 = 1 + 6, \dim H'_7 = 1$. (Most perfect squares in these formulas are dimensions of corresponding unitary groups, and the other numbers are dimensions of tori).
It implies that $\sum_{i=1}^7 (\dim G_i - \dim H'_i)
= 28 = 1 + 2 + 3 + 4 + 5 + 6 + 7$, which is exactly what one gets in the case of regular fibers, so the dimension of this
overlapping diamonds fiber is equal to the dimension of
regular fibers.

In general, the dimension of $\dim F^{-1}(\Lambda)$ by the following combinatorial formula. For simplicity, we will show how to deduce this explicit formula only for the case of generic coadjoint orbits; the more general case can be done in a similar way.

To each connected chain of equalities in the Gelfend-Cetlin triangle we associate 
a sequence of positive numbers
\begin{equation}
l_1,l_2,\hdots,l_s
\end{equation}
where $l_1 = 1$ is the number of eigenvalues in the intersection of the chain with the first line (from bottom up) for which this intersection is non-empty (and hence automatically contains exactly one eigenvalue), $l_2$ is the number of eigenvalues in the intersection of the chain with the next line, and so one, and $l_s =1$ is for the
last line with non-empty intersection with the chain. For example,
for the above example of overlapping diamonds, this sequence is
$1,2,3,2,3,2,1$. In general, for any two consecutive numbers, $l_{i-1}$  and $l_i$ 
in this chain, there are only 3 possibilities: $l_{i}-l_{i-1} = 1$ (increase by 1),
$l_{i}-l_{i-1}= 0$ (remains constant), and $l_{i}-l_{i-1} = -1$ (decrease).
Of course, for each increase there must a decrease. 

Notice now that if $l_{i}-l_{i-1} = 1$ then
\begin{equation}
\dim G_{(i)} - \dim G_{(i)}^{{xchain}} - (\dim H'_{(i)} - \dim H_{(i)}'^{xchain}) = l_i - 1 = l_{i-1} 
\end{equation}
where $G_{(i)}$ is the group $G_{i+s}$, where the line number $i+s$ has the relative index $i$ in the above sequence (i.e., the line number $s+1$ is the first line which has non-trivial intersection with our chain), and similarly for  $H'_{(i)}$,
$\dim G_{(i)}^{xchain}$ is the dimension of $G_{i+s}$ in the case when this chain of equalities disappears (but the other chains of equalities in the Gelfand-Cetlin traingles remain the same),
and similarly for $H_{(i)}'^{xchain}$.

On the other hand, if $l_{i}-l_{i-1}= - 1$ or $l_{i}-l_{i-1} = 0$ then
\begin{equation}
\dim G_{(i)} - \dim G_{(i)}^{xchain} - (\dim H'_{(i)} - \dim H_{(i)}'^{xchain})= - l_i,
\end{equation}

Now, if there is no equality of the type  $l_{i}-l_{i-1}$, then since for each
increase there is exactly one decrease, all of the above numbers add up to 0, i.e.
we have 
\begin{equation}
\sum_{i=1}^s (\dim G_{(i)} - \dim G_{(i)}^{xchain} - (\dim H'_{(i)} - \dim H_{(i)}'^{xchain})) = 0,
\end{equation}
For example, in the above example of two overlapping diamonds, the above sum is
$1 + 2 + (-2) + s + (-2) + (-1) = 0$. But in the more general case, when there are
equalities $l_{i}=l_{i-1}$, they will contribute negative terms to the sum, and we get:
\begin{equation}
\sum_{i=1}^s (\dim G_{(i)} - \dim G_{(i)}^{xchain} - (\dim H'_{(i)} - \dim H_{(i)}'^{xchain})) = - \sum_{l_i = l _{i-1}} l_i,
\end{equation}

The above discussions imply the following formula:

\begin{proposition}
\label{prop:dimension}
i) For any value $\Lambda$ of the momentum map of the Gelfand-Cetlin system on any 
coadjoint orbit $\orbit{\lambda}$ we have
\begin{equation}
\label{eqn:FiberDim}
\dim F^{-1}(\Lambda) = \dim F^{-1}(\Lambda_0) - \sum_{l_\alpha = l_{\alpha -1}} l_{\alpha} \leq \dim F^{-1}(\Lambda_0),
\end{equation}
where $\dim F^{-1}(\Lambda_0)$ means a regular Lagrangian fiber of the system, and
$l_\alpha$ in the sum runs through all the intersections of the connected chains
of equalities of eigenvalues with the horizontal lines in the Gelfand-Cetlin triangles  which satisfy the equation $l_{i}=l_{i-1}$ (with the above notations).

ii) The equality $\dim F^{-1}(\Lambda) = \dim F^{-1}(\Lambda_0)$ holds if and only if each connected chain of equalities of eigenvalues in the Gelfand-Cetlin triangle 
is axially symmetric (so that we have $l_{i} \neq l_{i-1}$ everywhere in every chain). In the case of a generic coadjoint orbit, it means that each chain is a union of symmetrically overlapping diamonds.

\begin{remark}
Proposition \ref{prop:dimension} was also obtained by
Cho--Kim--Oh \cite{CKO_GC2018} by a different method and using a different language.
\end{remark}

\end{proposition}

\section{Isotropic character of the fibers}

The dimension formula \eqref{eqn:FiberDim} shows that the dimension of every fiber is less or equal
to the dimension of a regular fiber, i.e. half the dimension of the coadjoint orbit, so it 
supports the conjecture that every fiber is isotropic. Indeed, we have:

\begin{proposition}
\label{prop:IsotropicGC}
All the fibers of the Gelfand-Cetlin system on any coadjoint orbit of $U(n)$ are isotropic submanifolds.
\end{proposition}

\begin{proof}
We will show a proof of the above proposition by direct computations, which are a bit technical but straightforward.

Lat $A \in \dim F^{-1}(\Lambda_0) \subset \orbit{\lambda} \subset \mathcal{H}(n)$ 
be a generic point in a given singular level set, and let $[Y,A]$ and $[Z,A]$ ($Y, Z \in \mathfrak{u}(n)$)
be two arbitrary tangent vectors at $A$ of $\orbit{\lambda}$ (by the notations of Section 2) which are tangent to
$F^{-1}(\Lambda_0)$ at $A$. By Formula \eqref{eqn:SymplecticForm} of the symplectic form, the isotropicness of
$\dim F^{-1}(\Lambda_0)$ means that 
\begin{equation}
\tr(A[Y,Z]) = 0.
\end{equation}
This is the equality that we will prove here. Notice that we can write $\tr(A[Y,Z]) = \tr (AYZ - AZY) = \tr (AYZ - YAZ)
= \tr ([A,Y]Z)$, where $[A,Y]$ is a tangent vector to the fiber  $\dim F^{-1}(\Lambda_0)$ at $A$.

Recall that transformation $A \mapsto CAC^* \in \dim F^{-1}(\Lambda_0)$ by a unitary matrix $C$ can be decomposed into a
composition of transformations by matrices $C_k \oplus I_{n-k}$ (where $C_k \in U(k)$ and $I_{n-k}$ is the identity matrix
of size $(n-k) \times (n-k)$) for $k=1,\hdots,n$ such that the conjugation of $A_k$ by $C_k$ leaves $A_{k-1}$ intact, where $A_k$
is the upper left submatrix of size $k$ of $A$, i.e. the upper left submatrix of size  $(k-1) \times (k-1)$
of $C_kA_kC_k^*$ is equal to $A_{k-1}$. At the level of tangent vector, it means that we have a decomposition
\begin{equation}
Y = \sum_{k=1}^n Y_k \oplus 0_{n-k}
\end{equation}
where $0_{n-k}$ means the zero square matrix of size $(n-k) \times (n-k)$, $Y_k \in \mathfrak{u}(k)$
such that the upper left submatrix of size  $(k-1) \times (k-1)$ of $[A_k,Y_k]$ is zero. 
The same decomposition
holds for $Z$: $Z = \sum_{k=1}^n Z_k \oplus 0_{n-k} $.  

Now, it is easy to sea immediately that if $k \neq h$ then $\tr ([A, Y_k \oplus 0_{n-k}] (Z_h \oplus 0_{n-h})) =0$.
So in order to show that $\tr ([A,Y]Z) =0$, we just need to show that $\tr ([A_k,Y_k]Z_k) =0$ for each $k$.
It is sufficient to show it for $k = n$, because for $k < n$ the proof will be exactly the same. So now we can
assume that $Y=Y_n$ and $Z=Z_n$ such that the upper left submatrices of size $(n-1) \times (n-1)$ of 
$[A,Y]$ and $[A,Z]$ are zero. 

By conjugating simultaneously $A,Y,Z$ by a unitary matrix of the type $C_{n-1} \oplus I_1$ which does not change
$\tr ([A,Y]Z)$ and the above condition on $[A,Y]$ and $[A,Z]$, we can reduce the problem to the case
when $A_{n-1} = \diag(\gamma_1,\hdots, \gamma_{n-1})$ is diagonal. In this case, we can write $A$ as: 
\begin{equation}
\label{eqn:A}
A = \begin{pmatrix}
\gamma_1 & & 0 & a_1 \\
         & \ddots & &  \vdots \\
0         &        & \gamma_{n-1} & a_{n-1} \\
\overline{a_1} & \hdots & \overline{a_{n-1}} & a_n       
\end{pmatrix}
\end{equation}
where $a_i = a_{in}$ and $a_{ni} = \overline{a_{in}} =
\overline{a_i}$. The condition that the upper left 
submatrix of size $(n-1) \times (n-1)$ of  $[A,Y]$
is equal to zero means that the entries $y_{ij}$
of $Y$ satisfy the following equations for all $i,j \leq n-1$:
\begin{equation}
\label{eqn:yij}
(\gamma_i - \gamma_j)y_{ij} = y_{in} \overline{a_j}
- a_i y_{nj}.
\end{equation}
In particular, when $i = j \leq n-1$ we have:
\begin{equation}
y_{in} \overline{a_i} = a_i y_{ni} = - a_i\overline{y_{in}},
\end{equation}
which means that $y_{in} \overline{a_i} \in \sqrt{-1}\mathbb{R}$ is pure imaginary. The same equalities hold for the entries of $Z = (z_{ij})$,
and in particular we have $z_{in} \overline{a_i} \in \sqrt{-1}\mathbb{R}$ for every $i \leq n-1$. 

By direct computations, we have

\begin{equation}
\tr ([A,Y]Z) = 2 \sqrt{-1} \Im ( \sum_{i=1}^{n-1}(a_i(y_{ii} - y_{nn}) - \gamma_i y_{in} + \sum_{j <n, j\neq i} y_{ij}a_j) \overline{z_{in}}),
\end{equation}
so the point is to show that 
\begin{equation}
\label{eqn:R1}
\left(\sum_{i=1}^{n-1}(a_i(y_{ii} - y_{nn}) - \gamma_i y_{in} + \sum_{j <n, j\neq i} y_{ij}a_j ) \overline{z_{in}}\right) \in \mathbb{R}.
\end{equation}

Recall that $A$ is supposed to be a generic element in
$\orbit{\gamma}$. If $A$ is generic but $a_i = 0$ then it means that $a_i = 0$ for all the other $A$ in $\orbit{\gamma}$ after a conjugation to the form \eqref{eqn:A}, which implies that the entry of index $(i,n)$ in $[A,Y]$ is zero, i.e. 
$\sum_{i=1}^{n-1}(a_i(y_{ii} - y_{nn}) - \gamma_i y_{in} + \sum_{j <n, j\neq i} y_{ij}a_j) = $, and so we can delete such terms (corresponding to the values of $i$ such that
$a_i = 0$) from \eqref{eqn:R1}. In other words, we just need to prove that
\begin{equation}
\label{eqn:R1b}
\left(\sum_{i <n, a_i \neq 0}(a_i(y_{ii} - y_{nn}) - \gamma_i y_{in} + \sum_{j <n, j\neq i} y_{ij}a_j ) \overline{z_{in}}\right) \in \mathbb{R}.
\end{equation}

Notice that when $a_{i} \neq 0$ then it follows from
$z_{in} \overline{a_i} \in \sqrt{-1}\mathbb{R}$
and $y_{in} \overline{a_i} \in \sqrt{-1}\mathbb{R}$
that we have $y_{in}\overline{z_{in}} \in \mathbb{R}$,
hence $\gamma_iy_{in}\overline{z_{in}} \in \mathbb{R}$
because $\gamma_i \in \mathbb{R}$. Moreover we have
$a_i(y_{ii} - y_{nn})\overline{z_{in}} \in \mathbb{R}$ because $a_i\overline{z_{in}}, y_{ii}, y_{nn} \in \sqrt{-1}\mathbb{R}$. Hence \eqref{eqn:R1} follows
from the following claim:
\begin{equation}
\label{eqn:R2}
\left(\sum_{i,j <n, a_i \neq 0,j\neq i} y_{ij}a_j\overline{z_{in}}\right) \in \mathbb{R}.
\end{equation}

Consider the terms in the sum in \eqref{eqn:R2}. If 
$\gamma_i \neq \gamma_j$ then by \eqref{eqn:yij}
we have $y_{ij} = \dfrac{y_{in} \overline{a_j}
- a_i y_{nj}}{\gamma_i - \gamma_j}$ and
\begin{equation}
\label{eqn:R3}
y_{ij}a_j\overline{z_{in}} =
\dfrac{1}{\gamma_i - \gamma_j} (y_{in} \overline{a_j}
- a_i y_{nj}) a_j\overline{z_{in}} =
\dfrac{1}{\gamma_i - \gamma_j} (\overline{a_j}a_j 
y_{in} \overline{z_{in}} + (a_i\overline{z_{in}})(a_j\overline{y_{jn}})) \in \mathbb{R}
\end{equation}
because $\gamma_i - \gamma_j \in \mathbb{R}$.
$y_{in} \overline{z_{in}} \in \mathbb{R}$,
$a_i\overline{z_{in}} \in \sqrt{-1}\mathbb{R}$
and $a_j\overline{y_{jn}} \in \sqrt{-1}\mathbb{R}$.

If $\gamma_i = \gamma_j$ then by symmetry between the indices $i$ and $j$, and the genericity of $A$, if $a_i \neq 0$ then we also have $a_j \neq 0$. It follows from \eqref{eqn:yij} that in this case we have 
$y_{in} \overline{a_j}
= a_i y_{nj} = - a_i \overline{y_{jn}}$, and similarly
$z_{in} \overline{a_j} =  - a_i \overline{z_{jn}}$.
Regrouping the terms of indices $(i,j)$ and $(j,i)$ in the sum \eqref{eqn:R3} together, we get:
\begin{equation}
y_{ij}a_j\overline{z_{in}}
+ y_{ji}a_i\overline{z_{jn}}
= y_{ij}a_j\overline{z_{in}} - y_{ji}
 \overline{a_j}z_{in} =
y_{ij}a_j\overline{z_{in}} +
\overline{y_{ij}} \overline{a_j}z_{in}
\in \mathbb{R}.
\end{equation}
Thus, the sum is \eqref{eqn:R2} is a real number, and we have finished the proof of Proposition \ref{prop:IsotropicGC}. 
\end{proof}

\begin{remark}
In \cite{CKO_GC2018} Cho, Kim and Oh also gave a proof of Proposition \ref{prop:IsotropicGC}. Their proof is based
on their description of each fiber of the Gelfand-Cetlin system as a tower of consecutive  fibrations,
and is also computational. 
\end{remark}

In relation to Proposition \ref{prop:IsotropicGC}, let us mention the following Proposition \ref{prop:isotropic}
about a sufficient condition for the isotropicness of
the singular fibers of analytic integrable systems. 
The Gelfand-Cetlin system is also analytic.
(The original momentum map of the Gelfand-Cetlin system
is not globally smooth, but it can be made into a globally analytic system by taking the symmetric 
functions of the
eigenvalues instead of the eigenvalues themselves).
Unfortunately, not all fibers of its satisfy this
sufficient condition, and since we don't have a more general proposition yet, we had to do a direct computational proof for the Gelfand-Cetlin case.

\begin{proposition}
\label{prop:isotropic}
Let $c \in \mathbb{R^n}$ be a singular value of the momentum map $G : (M^{2n},\omega) \to \mathbb{R}^n$ of a real analytic momentum map of an integrable Hamiltonian system, such that the fiber $G_\mathbb{C}^{-1}(c)$
 of a complexification $G_\mathbb{C} : M_\mathbb{C}^{2n} \to \mathbb{C}^n$ of $G$ is an 
analytic variety of dimension $n$ on which the set of singular points of the complexified momentum map
$G_\mathbb{C}$ is of positive codimension. ($M_\mathbb{C}^{2n} \supset M^{2n}$ is a local
complexification of $M_\mathbb{C}^{2n}$). Then we have:

i) If $G^{-1}(c)$ is a submanifold of $M^{2n}$
then it is isotropic. 

ii) If $G^{-1}(c)$ is a singular analytic variety then
the strata of its Whitney stratification are isotropic. 
\end{proposition}

\begin{proof}
Consider a complex singular level set $G_\mathbb{C}^{-1}(c) \supset G^{-1}(c)$ where
$c$ is a singular value of $G$. According to Whitney's stratification theorem
(see, e.g., \cite{Schwartz_Stratification1966}) and the above conditions, 
$G_\mathbb{C}^{-1}(c)$ admits a stratification, whose highest dimensional strata
are of dimension $n$ and are regular with respect to the map $G_\mathbb{C}$, hence
they are Lagrangian. By Whitney's regularity condition (b), it implies by induction
that all the smaller-dimensional strata are isotropic. 
In the case when $G^{-1}(c)$ is regular, the set of points of $G^{-1}(c)$  which admits an open 
neighborhood in $G^{-1}(c)$ lying entirely in a 
stratum is a dense subset of $G^{-1}(c)$, and at those points we have that $G^{-1}(c)$ is isotropic. 
By continuity, $G^{-1}(c)$ is isotropic everywhere. The case of when  $G^{-1}(c)$ is a singular variety is
similar, and in that case every stratum of $G^{-1}(c)$ is isotropic. 
\end{proof}

\section*{Acknowledgements}

Some results of this paper (sometimes with errors or imprecisions) 
were announced earlier by the authors at various talks and conferences during many years. 
We would like to thank the participants of these talks and conferences for their interest, 
and especially Yael Karshon, Jeremy Lane and Tudor Ratiu for interesting discussions. 

This paper was written by the authors mainly in 2017, 
and some of the results were included in Damien Bouloc's PhD thesis 
defended in June 2017. We thank the Centre de Recerca Matematica-CRM for hospitality during the visit of Damien Bouloc to CRM in June 2017.

Some parts of the paper were written or revised by N.T. Zung 
during his stay at 
the Center for Geometry and Physics, IBS, Pohang (Korea) in early 
2018, and he would like to thank this Center and its director 
Yong-Geun Oh for the invitation and excellent working conditions. He would like also to thank Yong-Geun Oh very
interesting discussions on the Gelfand-Cetlin system and related subjects, including toric degenerations and mirror
symmetry. 

\bibliographystyle{alpha}
\bibliography{biblio.bib}

\end{document}